\RequirePackage{fix-cm}
\documentclass[smallextended]{svjour3}
\smartqed
\usepackage{multirow,mathtools}
\usepackage{amsmath,amssymb}
\usepackage{amsfonts}
\usepackage{url}
\usepackage{color}
\usepackage[pointedenum]{paralist}
\usepackage{listings}
\usepackage{cases}
\usepackage{tikz}
\usepackage{graphicx}
\usepackage{subcaption}
\usepackage{grffile}
\usepackage[breaklinks=true,pdfstartview=FitH]{hyperref}
\usepackage{algorithm}
\usepackage{algorithmic}

\newcommand{\TheTitle}{Inexact Variable Metric Stochastic Block-Coordinate Descent for Regularized Optimization}
\newcommand{\RunningTitle}{Inexact Variable Metric Stochastic Proximal
Block-Coordinate Descent}

\titlerunning{\RunningTitle}

\title{{\TheTitle}\thanks{Version of \today.}
	\thanks{This work was supported by NSF awards 1447449, 1628384,
		1634579, and 1740707;
		Subcontracts 3F-30222 and 8F-30039 from Argonne National
		Laboratory; and
	Award N660011824020 from the DARPA Lagrange Program.
This work was done mostly when Ching-pei Lee was at the University of
Wisconsin-Madison.}
}

\author{
Ching-pei Lee
\and
Stephen~J. Wright
}

\institute{Ching-pei Lee \at
	\email{leechingpei@gmail.com}\\
	Department of Mathematics \& Institute for Mathematical Sciences,
	National University of Singapore, Singapore\\
	Stephen~J. Wright \at
	\email{swright@cs.wisc.edu}\\
	Computer Sciences Department and Wisconsin Institute for
	Discovery, University
	of Wisconsin-Madison, Madison, WI, USA
}
\date{received: date / accepted: date}
\newcommand{\citet}[1]{\cite{#1}}
\newcommand{\citep}[1]{\cite{#1}}
\journalname{}
\begin{document}

\newcommand{\R}{\mathbb{R}}
\newcommand{\N}{\mathbb{N}}
\newcommand{\diag}{\mbox{\rm diag}\,}
\newcommand{\cP}{{\mathcal P}}
\newcommand{\cA}{{\mathcal A}}
\newcommand{\cH}{{\mathcal H}}
\newcommand{\cL}{{\mathcal L}}
\newcommand{\cM}{{\mathcal M}}
\newcommand{\cN}{{\mathcal N}}
\newcommand{\E}{\mathbb{E}}
\def\Lmax{L_{\max}}
\def\Lmin{L_{\min}}
\def\Lavg{L_{\text{avg}}}
\def\AA{\mathcal{A}}
\def\HH{\mathcal{H}}

\maketitle

\begin{abstract}
Block-coordinate descent (BCD) is a popular framework for large-scale
regularized optimization problems with block-separable structure.
Existing methods have several limitations. They often assume that
subproblems can be solved exactly at each iteration, {which in
  practical terms usually restricts the quadratic term in the
  subproblem to be diagonal, thus} losing most of the benefits of
higher-order derivative information. Moreover, in contrast to the
smooth case, non-uniform sampling of the blocks has not yet been shown
to improve the convergence {rate bounds} for regularized
problems. This work proposes an inexact randomized BCD method based on
a regularized quadratic subproblem, in which the quadratic term can
vary from iteration to iteration: a ``variable metric''. We provide a
detailed convergence analysis for both convex and nonconvex problems.
Our analysis generalizes {to the regularized case Nesterov's
  proposal to improve convergence of BCD} by sampling proportional to
the blockwise Lipschitz constants. We improve the convergence rate in
the convex case by weakening the dependency on the initial objective
value. Empirical results also show that significant benefits accrue
from the use of a variable metric.
\end{abstract}

\section{Introduction} \label{sec:intro}
We consider the following regularized minimization problem:
\begin{equation}
	\min_{x} \, F\left(x\right) \coloneqq f\left(x\right) +
	\psi\left(x\right),
	\label{eq:f}
\end{equation}
where $f$ is blockwise Lipschitz-continuously differentiable (defined
below) but not necessarily convex, and the regularizer $\psi$ is
convex, extended-valued, proper, closed, and block-separable, but
possibly nondifferentiable.  We assume $F$ is lower-bounded and denote
the solution set by $\Omega$, which is assumed to be nonempty.  For
simplicity, we assume $x \in \R^n$, but our methods can be applied to
matrix variables too. We decompose $x\in \R^n$ into $N$ blocks such
that
\begin{equation*}
x = (x_1,x_2,\dotsc,x_N) \in \R^n, \quad
x_i \in \R^{n_i},\quad n_i \in \N, \quad \sum_{i=1}^N
n_i = n,
\end{equation*}
and assume throughout that the function $\psi$ can be decomposed as
\[
\psi(x) = \sum_{i=1}^N \psi_i(x_i),
\]
where all $\psi_i$ have the properties claimed for $\psi$ above.  Many
regularized empirical risk minimization (ERM) problems in machine
learning have this structure with $n_i > 1$ for all $i$; see, for
example, \cite{YuaL06a,MeiVB08a,CraS02a,LebL02a,TsoJHA05a,LeeL13a}.
For the block-separability of $x$, we use the column submatrices of
the identity denoted by $U_1,U_2,\dotsc, U_N$, where {$U_i \in \R^{n
\times n_i}$} corresponds to
the indices in the $i$th block of $x$. Thus, we have
\begin{equation*}
	x_i = U_i^\top x, \quad x = \sum_{i=1}^N U_i x_i, \quad \mbox{and} \;
\nabla_i f = U_i^\top \nabla f.
\end{equation*}
The blockwise Lipschitz-continuously differentiable property is that
there exist constants $L_i > 0$, $i=1,2,\dotsc,N$, such
that\footnote{{We use the Euclidean norm throughout the
paper.}}
\begin{equation}
\left\|\nabla_i f(x + U_i h) - \nabla_i f(x)\right\| \leq L_i
\|h\|,\quad \forall h \in \R^{n_i},\quad \forall x \in \R^n.
\label{eq:compL}
\end{equation}

We consider randomized block-coordinate-descent (BCD) type methods to
optimize \eqref{eq:f}, where only one block of variables is updated
{at each iteration}.  Moreover, we define subproblems with
varying quadratic terms, and use possibly non-uniform sampling to
select the block to be updated.  To accommodate general quadratic
terms and complicated regularizers $\psi_i$, we also allow inexactness
in computation of the update step.

The $k$th iteration of the ``exact'' version of our approach proceeds
as follows.  Given the current iterate $x^k$, we pick a block $i$,
according to some discrete probability distribution over
$\{1,\dotsc,N\}$, and minimize a quadratic approximation of $f$ plus
the function $\psi_i$ for that block, to obtain the update
direction $d^k_i$. That is, we have
\begin{equation}
\label{eq:quadratic}
d^{k*}_i \coloneqq \arg\min_{d_i\in \R^{n_i}} \, Q^k_{i}( d_i),
\end{equation}
where
\begin{align}
\label{eq:Qdef}
Q^{k}_{i}( d_i ) \coloneqq \nabla_i f\left( x^k
\right)^\top d_i + \tfrac12 d_i^\top H^k_i d_i
+ \psi_i\left( x^k_i + d_i
\right) - \psi_i\left( x^k_i \right),
\end{align}
and $H^k_i\in\R^{n_i\times n_i}$ is some positive-definite matrix that
can change over iterations.  A backtracking line search along $d^k_i$
is then performed to determine the step.

Methods of this type have been discussed in existing works
\citep{Nes12a,TapRG16a,FouT15a}, but under various assumptions that
may be impractical for some problems.  {In \cite{Nes12a}, it is
  required that the component-wise Lipschitz constants are known, and
  that} \eqref{eq:quadratic} is solved to optimality, which is usually
possible only when $\psi_i$ possesses some simple structure and each $H^k_i$
is diagonal.  In \cite{TapRG16a}, the matrices $H^k_i$ are required
to be fixed over
iterations. The extension described in \cite{FouT15a} is close to our
framework, but (as they point out) their subproblem termination
condition may be expensive to check except for specific choices of
$\psi$. By contrast, we aim for more general applicability by
requiring only that \eqref{eq:quadratic} is solved inexactly, in a
sense defined below in \eqref{eq:approx}, that does not even need to
be checked. Moreover, these works consider only uniform sampling for
the regularized problem (for which $\psi \not\equiv 0$).\footnote{For
    the special case $\psi \equiv 0$, works including \cite{TapRG16a}
    considered arbitrary samplings.}  Since \cite{Nes12a} showed
  possible advantages of non-uniform sampling in the non-regularized
  case, we wish to consider non-uniform sampling in the regularized
  setting too.  Others studied the cyclic version of the
  block-coordinate approach under various assumptions
  \citep{ChoPR16a,SunH15a,TseY09a,Yun14a}. (The cyclic variant is
  significantly slower than the randomized one in the worst case
  \citep{SunY16a}.)

This paper contributes to both theory and practice.  From the
practical angle, we extend randomized BCD for regularized functions to
a more flexible framework, involving variable quadratic terms and line
searches, recovering existing BCD algorithms as special cases.
Knowledge of blockwise Lipschitz constants is not assumed.  Our
algorithms are thus more practical, applicable to wider problem
classes (including nonconvex ones), and significantly faster in
practice.  The theoretical contributions are as follows.
\begin{enumerate}
\item For convex problems, our analysis reflects a phenomenon that is
  widely observed in practice for BCD on convex problems: {a kind of}
  Q-linear convergence in the early stages of the algorithm, until a
  modest degree of suboptimality is attained.  This result can be used
  to strongly weaken the dependency of the iteration complexity on the
  initial objective value.
\item We show that global linear convergence holds under {
  the quadratic growth condition, which is significantly weaker than
  strong convexity.}
\item Our convergence analysis allows arbitrary sampling probabilities
  for the blocks, and we show that non-uniform distributions can
  reduce the iteration complexity significantly in some cases.
\item Inexactness in the subproblem solution affects the
  {bounds on the number of iterations of the main algorithm} in
  a benign way. It follows that if approximate solutions can be
  obtained cheaply for the subproblems, overall running time of the
  algorithm can be reduced significantly.
\end{enumerate}
Special cases of our algorithm of diagonal $H$ extend existing
analysis for { regularized problems, showing that for the
  regularized problem \eqref{eq:f},}
sampling with probability proportional to the value of the blockwise
Lipschitz constants $L_i$ enjoys the same improvement of the iteration
bound as the non-regularized case, by a factor of $\Lmax/\Lavg$ over
uniform sampling, where
\begin{equation} \label{eq:Lminmax}
	\Lmax \coloneqq \max_{1 \leq i \leq N} \, L_i,\quad \Lavg \coloneqq
	\frac{1}{N} \sum_{i=1}^N L_i,\quad  \Lmin \coloneqq \min_{1
	\le i \le N}\, L_i.
\end{equation}
(We believe this result to be novel in the regularized setting
\eqref{eq:f}.)  The same sampling strategy produces similar advantages
for nonconvex problems, an observation that is novel even for the
non-regularized case.

We introduce our assumptions and the proposed algorithm in Section
\ref{sec:algorithm}.  Section \ref{sec:analysis} provides detailed
convergence analysis for various classes of problems, including
nonconvex problems and problems for which our algorithm enjoys global
linear convergence.  The special case of traditional BCD {
  (in which $H^k_i$ are multiples of identity matrices) with}
non-uniform sampling is studied in Section~\ref{sec:rcd}.  We discuss
related works in Section~\ref{sec:related} and efficient
implementation of our algorithm for a wide class of problems in
Section~\ref{sec:erm}.  Computational results are shown in
Section~\ref{sec:exp}, with concluding remarks in
Section~\ref{sec:conclusion}.

\section{Proposed Algorithm}
\label{sec:algorithm}

We focus throughout this paper on the case in which
\eqref{eq:quadratic} is difficult to solve in closed form, so is
solved inexactly by an iterative method, such as coordinate descent,
proximal gradient, or their {respective} accelerated variants.
We assume that $d^k_i$ is an $\eta$-approximate solution to
\eqref{eq:quadratic}, for some $\eta \in [0,1)$ fixed over all $k$ and
  all $i$, satisfying the following condition:
\begin{equation}
{-\eta Q^{k*}_i = }\eta \left(Q^k_i\left(0\right) -  Q^{k*}_i \right) \geq
Q^k_i\left(d^k_i\right) - Q^{k*}_i,
\label{eq:approx}
\end{equation}
where $Q^{k*}_i \coloneqq \inf_{d_i} Q^k_i(d_i) {=
  Q^k_i(d^{k*}_i)}$. Note that the setting $\eta=0$ corresponds to the
special case in which the subproblems are solved exactly. In general,
we do not need to know the value of $\eta$ or to verify the condition
\eqref{eq:approx} explicitly; we merely need to know that such a value
exists. For example, if the algorithm used to solve
\eqref{eq:quadratic} has a global Q-linear convergence rate, and if we
run this method for a fixed number of iterations, then we know that
\eqref{eq:approx} is satisfied for {\em some} value {$\eta
  \in [0,1)$}, even if we do not know this value explicitly. Further
  discussions on how to achieve this condition can be found in, for
  example, \cite{BonLPP16a,LeeW18a}.  Our analysis can be extended
  easily to variable, adaptive choices of $\eta$, which might lead to
  better iteration complexities, but for the sake of interpretability
  and simplicity, we fix $\eta$ {independent of $k$ and $i$} in
  our discussion throughout.

{Our algorithm is summarized as Algorithm~\ref{alg:icd}.}
At the current iterate $x^k$, a block $i_k$ is chosen according to some
discrete probability distribution over $\{1,2,\dotsc,N\}$, with
{{\em strict positive} probabilities} $p^k_1,p^k_2,\dotsc,p^k_N$.
For the selected block $i_k$, we compute the partial gradient
$\nabla_{i_k} f$ and choose a positive-definite $H^k_{i_k}$, thus
defining the subproblem {objective} \eqref{eq:Qdef}.  The
selection of $H^k_{i_k}$ is application-dependent; possible choices
include the (generalized) Hessian, \footnote{Since $\nabla_{i_k} f$ is
  Lipschitz continuous, it is differentiable almost everywhere.
  Therefore, we can at least define a generalized Hessian as suggested
  by \cite{HirSN84a}.}  its quasi-Newton approximation, and a diagonal
approximation to the Hessian. A diagonal damping term may also be
added to $H^k_{i_k}$. After finding an approximate solution $d^k_{i_k}$ to
\eqref{eq:quadratic} that satisfies \eqref{eq:approx} for some $\eta
\in [0,1)$, we conduct a backtracking line search, as in
  \cite{TseY09a}: Given $\beta, \gamma \in (0,1)$, we let
  $\alpha^k_{i_k}$ be the largest value in $\{1,\beta^1,\beta^2,
  \dotsc\}$ such that the following sufficient decrease condition is
  satisfied:
\begin{equation}
	F\left(x^k + \alpha^k_{i_k} U_{i_k} d_{i_k}^k \right) \leq F\left(
	x^k \right) +
	\alpha_{i_k}^k \gamma \Delta^k_{i_k},
\label{eq:armijo}
\end{equation}
where
\begin{equation}
\Delta^k_i \coloneqq \nabla_i f\left( x^k \right)^\top d^k_i + \psi_{i}\left(
x^k_{i} + d^k_i \right) - \psi_{i}\left( x^k_i \right).
\label{eq:Delta}
\end{equation}
Then the iterate is updated to $x^k + \alpha^k_{i_k} U_{i_k} d_{i_k}^k$.

\begin{algorithm}[tb]
\caption{Inexact variable-metric block-coordinate descent for
	\eqref{eq:f}}
\label{alg:icd}
\begin{algorithmic}[1]
\STATE Given $\beta, \gamma \in (0,1)$, $\eta \in [0,1)$, and $x^0
	\in \R^n$;
\FOR{$k=0,1,2,\dotsc$}
	\STATE Pick a probability distribution
	{$p^k_1,\dotsc,p^k_N > 0$,
		$\sum_{i} p^k_i = 1$}, and sample {${i_k}$} accordingly;
	\STATE Compute $\nabla_{i_k} f(x^k)$ and choose a
	positive-definite $H^k_{i_k}$;
	\STATE Approximately solve \eqref{eq:quadratic} {for $i=i_k$}  to obtain a
	solution $d^k_{i_k}$ satisfying \eqref{eq:approx};
	\STATE Compute $\Delta_{i_k}$ by \eqref{eq:Delta}, {with
	$i=i_k$}; Set {$\alpha_{i_k}^k \leftarrow 1$};
	\WHILE{\eqref{eq:armijo} is not satisfied}
	\STATE $\alpha^k_{i_k} \leftarrow \beta \alpha^k_{i_k}$;
	\ENDWHILE
	\STATE $x^{k+1} \leftarrow x^k + \alpha^k_{i_k} U_{i_k} d^k_{i_k}$;
\ENDFOR
\end{algorithmic}
\end{algorithm}

\section{Convergence Analysis}
\label{sec:analysis}

Our convergence analysis extends that of \cite{LeeW18a}, which can be
considered as a special case of our framework {in which there
  is just one block $(N=1$)}.  Nontrivial modifications are needed to
allow for multiple blocks and non-uniform sampling. {In the
  results of this section we often focus on a particular iteration
  $k$, but rather than considering the consequences of updating the
  chosen block $i_k$ at that iteration, we examine what would happen
  for {\em all possible} choices of $i=1,2,\dotsc,N$, if each of these
  values happened to be chosen as $i_k$. Since the actual update block
  $i_k$ is chosen randomly from among these $N$ possibilities, we
  obtain results about the expected change in $F$ by taking {\em
    expectations} over all these hypothetical choices.}

{The following result tracks \cite[Corollary~4]{LeeW18a} and
  its proof is therefore omitted. Note that we focus on iteration $k$,
  and obtain lower bounds for each possible choice of update block
  $i=1,2,\dotsc,N$.}
\begin{lemma}
\label{lemma:linesearch}
{At the $k$th iteration, suppose that $H^k_i \succeq m_i I$ for
  some $m_i > 0$, for $i=1,2,\dotsc,N$, and that the subproblem
    solution $d^k_i$ satisfies \eqref{eq:approx}, for each
    $i=1,2,\dotsc,N$.} Then we have
\begin{equation}
\Delta^k_i \leq -\frac{1}{1 + \sqrt{\eta}} \left(d^k_i\right)^\top
H^k_i d^k_i
\leq -\frac{m_i}{1 + \sqrt{\eta}} \|d_i^k\|^2.
\label{eq:deltabound2}
\end{equation}
Moreover, the backtracking line search procedure in
Algorithm~\ref{alg:icd} terminates finitely, with the step size
$\alpha^k_i$ lower bounded by
\begin{equation}
\alpha^k_i \geq \bar\alpha_i \coloneqq \min\left\{1, \frac{2\beta\left(1 -
\gamma\right) m_i}{L_i\left(1 + \sqrt{\eta}\right)}\right\}.
\label{eq:linesearchbound2}
\end{equation}
\end{lemma}

The bound $\bar\alpha_i$ in \eqref{eq:linesearchbound2} is a
worst-case guarantee.  For properly selected $H^k_i$ (for example, when
$H^k_i$ includes true second-order information about $f$ confined to the
$i$th block), the steps will usually be closer to $1$ because the last
inequality in \eqref{eq:deltabound2} is typically loose.

We proceed to deal with the cases in which $F$ is convex and not
necessarily convex, respectively.

\subsection{Convex Case}

We {first state the optimal set strong convexity condition,
  proposed in \cite{LeeW18a}, that will be used in} showing global
linear convergence of Algorithm~\ref{alg:icd}.
\begin{definition}
Given any function $F$ whose minimum value $F^*$ is attainable,
and for any $x$, define $P_{\Omega}(x)$ to be the
{(Euclidean-norm)} projection of $x$ onto the optimal set
$\Omega$.
We say that $F$ satisfies the {\em optimal set strong convexity} {(OSSC)}
condition with parameter $\mu \geq 0$, if for any $x$ and any
$\lambda \in [0,1]$, the following holds.
\begin{align}
  \nonumber
F(\lambda x + & (1 - \lambda )
P_{\Omega}(x)) \\
\label{eq:strong}
& \leq \lambda F\left(x\right) + \left(1 - \lambda\right) F^* -
\tfrac12 \mu \lambda \left( 1 - \lambda \right)  \left\|x -
P_{\Omega}\left(x\right)\right\|^2.
\end{align}
\end{definition}

The following technical lemma is crucial for both the convergence rate
proofs and for motivating the choice of $p_i$, $i=1,2,\dotsc,N$.
{We will use this result to bound the expected improvement of
  the objective value over one step, which leads to convergence rates
  for the algorithm.}
\begin{lemma}
\label{lemma:Q}
{Let $f$ and $\psi$ be convex with $F$ satisfying
  \eqref{eq:strong} for some $\mu \geq 0$. At iteration $k$, we
  consider matrices $H^k_i \succeq 0$ with $H_i^k \in \R^{n_i \times
    n_i}$, $i=1,\dotsc,N$, probability distribution $\{p_i^k\}_{i=1}^N
  > 0$, and step sizes $\{\alpha^k_i\}_{i=1}^N > 0$. We define}
\begin{align*}
\cP_k &\coloneqq \diag(p^k_1 I_{n_1},\dotsc, p^k_N I_{n_N}),\quad
\AA_k \coloneqq \diag(\alpha_1^k
I_{n_1},\dotsc, \alpha_N^k I_{n_N}),\\
\HH_k &\coloneqq \diag(H_1^k,\dotsc,H_N^k).
\end{align*}
Then for $Q_i^k$ defined by \eqref{eq:Qdef}, the following holds
for all $\lambda \in \left[0,1\right]$ and all $\theta$ such that
$0\leq\theta\leq \alpha^k_i p^k_i$, $i=1,\dotsc,N$:
\begin{align}
\nonumber
\E_i\left[\left.\alpha^k_i Q^{k*}_i \right| x^k\right]
&\leq \theta\lambda \left(F^* - F\left(x^k\right) \right)- \tfrac12 \mu \theta
	\lambda \left( 1 - \lambda \right) \left\|x^k -
P_\Omega\left(x^k\right)\right\|^2 +\\
& \tfrac12 \theta^2\lambda^2 \left( x^k - P_{\Omega}\left( x^k \right)
\right)^\top  \cP_k^{-1}\AA_k^{-1} \HH_k\left( x^k - P_{\Omega}\left(
x^k \right)
\right).
\label{eq:Q}
\end{align}
\end{lemma}
\begin{proof}
Given any $d\in \R^n$, let $\tilde d \coloneqq \AA_k \cP_k d \in
\R^n$.  We obtain by change of variables that
\begin{align}
\nonumber
&~\E_i{\left[\left.\alpha_i^k Q_i^{k*} \right| x^k \right]} \\
\nonumber
=&~ \min_d \,\nabla f \left( x^k
	\right)^\top  \AA_k \cP_k d + \tfrac12 d^\top  \HH_k \AA_k \cP_k d + \sum_{i=1}^N
	\alpha_i^k p_i^k \left(\psi_i\left(x_i^k + d_i\right) -
	\psi_i\left(x_i^k\right)\right)\\
\nonumber
=&~ \min_{\tilde d \in \R^n} \, \nabla f\left( x^k \right)^\top  \tilde d + \tfrac12
	\tilde{d}^\top  \cP_k^{-1} \AA_k^{-1} \HH_k \tilde d + \sum_{i=1}^N
	\alpha_i^k p_i^k \left( \psi_i \left( x^k_i +
	\frac{\tilde{d}_i}{\alpha_i^k p_i^k} \right) - \psi_i\left( x_i^k
	\right)\right)\\
\nonumber
\leq&~ \min_{\tilde d \in \R^n} \,\min_{\theta\in[0,1] \, \text{s.t.} \,
	\frac{\theta}{\alpha_i^k p_i^k} \leq 1,\forall i}
	\nabla f\left( x^k \right)^\top  \left(\theta\tilde d\right) + \tfrac12
	\left(\theta\tilde{d}\right)^\top  \cP_k^{-1}\AA_k^{-1}  \HH_k
	\left(\theta\tilde d\right)\\
\label{eq:intermediate1}
	&\qquad\qquad + \sum_{i=1}^N \alpha_i^k p_i^k \left(
	\psi_i \left( x_i^k + \frac{\theta\tilde{d}_i}{\alpha_i^k p_i^k} \right) -
	\psi_i\left( x_i^k \right)\right),
\end{align}
{where each $\tilde d_i \in \R^{n_i}$. In
  \eqref{eq:intermediate1}, we used the fact that $\theta \tilde {d}$
  is also a feasible point for the left-hand side, hence its objective
  value is no smaller than the minimizer.}

Next, from the convexity of $f$, we have
\begin{equation*}
\nabla f\left( x^k \right)^\top  \theta \tilde d = \theta \left( \nabla
f\left( x^k \right)^\top  \tilde d
\right) \leq
\theta \left( f\left( x^k + \tilde d \right) - f\left( x^k \right)
\right),
\end{equation*}
and from  $\theta / (\alpha_i^k p_i^k) \leq 1$ for all $i$ and
the convexity of $\psi$, we obtain
\begin{align*}
	\psi_i \left( x^k_i + \frac{\theta \tilde{d}_i}{\alpha_i^k p_i^k} \right)
	&\leq \left( 1 - \frac{\theta }{\alpha_i^k p_i^k} \right)
	\psi_i\left( x_i^k \right) + \frac{\theta}{\alpha_i^k p_i^k}
	\psi_i\left( x_i^k + \tilde{d}_i \right)\\
	&= \frac{\theta }{\alpha_i^k p_i^k}\left( \psi_i \left( x_i^k +
	\tilde{d}_i \right) - \psi_i\left( x_i^k \right) \right) + \psi_i \left(
	x_i^k \right).
\end{align*}
Therefore, we have
\begin{align}
\nonumber
&~\min_{\tilde d}\,\min_{\theta\in[0,1] \, \text{s.t.} \,
	\frac{\theta}{\alpha_i^k p_i^k} \leq 1,\forall i}
	\nabla f\left( x^k \right)^\top  \left(\theta\tilde d\right) + \tfrac12
	\left(\theta\tilde{d}\right)^\top  \cP_k^{-1}\AA_k^{-1}  \HH_k
	\left(\theta\tilde d\right) \\
	\nonumber
	&\qquad\qquad+ \sum_{i=1}^N \alpha_i^k p_i^k \left(
	\psi_i \left( x^k_i + \frac{\theta\tilde{d}_i}{\alpha_i^k p^k_i} \right) -
	\psi_i\left( x_i^k \right)\right)\\
\nonumber
\leq&~ \min_{\tilde d} \, \min_{\theta\in[0,1] \, \text{s.t.} \,
	\frac{\theta}{\alpha_i^k p^k_i} \leq 1,\forall i}
	\theta\left(F(x^k+\tilde d) - F(x^k)\right)  + \frac{\theta^2}{2}
	\tilde{d}^\top  \cP_k^{-1}\AA_k^{-1} \HH_k \tilde{d}\\
	\nonumber
\leq&~ \min_{\lambda\in [0,1]}\,\min_{\theta\in[0,1] \, \text{s.t.} \,
	\frac{\theta}{\alpha^k_i p^k_i} \leq 1,\forall i}
\theta \left(F\left(x^k + \lambda \left(
	P_{\Omega}\left(x^k\right)  - x^k \right) \right) - F\left( x^k
	\right)\right) \\
	&\qquad\qquad + \frac{\theta^2
	\lambda^2}{2}\left(P_\Omega\left(x^k\right)  - x^k\right)^\top 
	\cP_k^{-1}\AA_k^{-1} \HH_k \left(P_\Omega\left(x^k\right) - x^k \right).
\label{eq:intermediate2}
\end{align}
The result \eqref{eq:Q} then follows from combining
\eqref{eq:intermediate1}, \eqref{eq:intermediate2}, and
\eqref{eq:strong}.  \qed
\end{proof}

By positive semidefiniteness of $H^k_i$ for all $i$ and all $k$,
\eqref{eq:armijo} implies that
\begin{align}
\nonumber
F\left(x^k + \alpha_{i_k}^k U_{i_k} d_{i_k}^k\right) - F\left( x^k \right)
&\leq \gamma \alpha_{i_k}^k \left(\Delta_{i_k}^k + \tfrac12
\left(d^k_{i_k}\right)^\top  H_{i_k}^k d_{i_k}^k\right) \\
& = \gamma \alpha_{i_k}^k Q_{i_k}^k(d_{i_k}^k) \leq (1 - \eta)\gamma \alpha^k_{i_k}  Q^{k*}_{i_k}.
\label{eq:usage}
\end{align}
Thus Lemma~\ref{lemma:Q} can be applied to the right-hand side of this
bound to obtain an estimate of the decrease in $F$ at the current
step.

Given any $x^0$, we define
\begin{equation}
R_0 \coloneqq \sup_{x: {F( x ) \leq F( x^0)}}\quad\left\|x - P_\Omega(x)\right\|.
\label{eq:r0}
\end{equation}
For the case of general convex problems, we make the assumption that
for any $x^0$, the value of $R_0$ defined in \eqref{eq:r0} is
finite.
{We are ready to state results concerning the rate of
  convergence.  Part 1 of the following result shows that when the
  objective function optimality gap $F(x^k)-F^*$ is above a certain
  threshold, a linear convergence rate applies. Part 2 identifies an
  iteration $k_0$ such that for $k \ge k_0$, and for a fixed probability
  distribution $\{ p_i \}$ for the choice of index to update, a
  sublinear ``$1/k$'' convergence rate applies.  Part 3 shows that
  when a fixed probability distribution $\{ p_i \}$ is used
  throughout, an initial linear phase of decrease in the expected
  objective function optimality gap is followed by a $1/k$ sublinear
  phase, and the change point of the phase is based on
  the expected value of $F(x^k) - F^*$ instead, making the iteration
  complexity calculable.}
\begin{theorem}
\label{thm:sub}
Assume that $f$ and $\psi$ are convex and \eqref{eq:compL} holds.
Suppose that at all iterations $k$ of Algorithm~\ref{alg:icd},
{and for any choice $i=i_k$ of the update block at iteration
  $k$, we have that \eqref{eq:approx} is satisfied with a fixed $\eta
  \in [0,1)$, with $H^k_i$ chosen such that}
\begin{equation}
H^k_i \succeq m_i I,\quad k=0,1,\dotsc,
\label{eq:H}
\end{equation}
for some $ m_i > 0$ for all $i$. Then the following are true.
\begin{enumerate}
\item At iteration $k$, given any probability distribution
  $\{p^k_i\}_{i=1}^N > 0$ for choosing the update block $i_k$, denote
  by $\{\alpha^k_i\}_{i=1}^N > 0$ the step sizes generated by the
  backtracking line search for each possible choice
  $i=1,2,\dotsc,N$. (These step sizes are guaranteed to be bounded
  away from zero, by Lemma~\ref{lemma:linesearch}.) Define
\begin{equation} \label{eq:pik}
	\pi^k \coloneqq \min_{1 \le i \le N} \alpha^k_i p^k_i,
\end{equation}
{and let $\cP_k$, $\AA_k$, and $\HH_k$ be defined as in
Lemma~\ref{lemma:Q}.} If
\[
  F\left( x^k \right) - F^* \geq \left(x^k -
  P_\Omega\left(x^k\right)\right)^\top  \cP^{-1}_k \AA_k^{-1}
  \HH_k\left(x^k - P_\Omega\left(x^k\right)\right) \pi^k,
  \]
  {then the expected improvement in objective optimality gap
    at this iteration is bounded away from $1$, as follows:}
\begin{align}
\frac{\E_{i_k}\left[F\left(x^{k+1}\right) -
F^*\right|\left.x^k\right]}{\left(F\left( x^k \right) - F^*\right)}
\leq \left( 1 - \frac{\left( 1 - \eta \right) \gamma \pi^k }{2}\right) .
\label{eq:earlylinear}
\end{align}
\item {Given $M_i \geq m_i, i=1,\dotsc,N,$}
and define
\begin{equation} \label{eq:def.M}
\cM \coloneqq \diag(M_1 I_{n_1},\dotsc, M_N I_{n_N}),\quad \bar \AA
\coloneqq \diag(\bar{\alpha}_1 I_{n_1},\dotsc, \bar{\alpha}_N
I_{n_N}),
\end{equation}
 where $\bar
\alpha_i$ are defined in Lemma \ref{lemma:linesearch}.  Given a
probability distribution $\{p_i\}_{i=1}^N > 0$, define
\begin{equation} \label{eq:barpi}
\cP \coloneqq \diag(p_1 I_{n_1},\dotsc, p_N I_{n_N}),\quad
  \bar\pi
  \coloneqq \min_{1 \le i \le N} \, \bar\alpha_i p_i,
\end{equation}
and let
\begin{equation}
	k_0 \coloneqq \arg \min\left\{ k: F\left( x^k \right) - F^* <
		\|\cP^{-1} \bar{\AA}^{-1} \cM\| \bar\pi R_0^2\right\}.
\label{eq:k0}
\end{equation}
{Suppose that for all $k \geq k_0$, the sampling of $i_k$
  follows the distribution $\{p_i\}$, which does not depend on $k$,
  and }
\begin{equation}
M_i I \succeq H^k_i \succeq m_i I,\quad i=1,\dotsc,N.
\label{eq:H2}
\end{equation}
Then for $k \geq k_0$, the expected objective follows a sublinear
convergence rate, as follows:
\begin{align}
	\E_{i_{k_0},i_{k_0+1},\dotsc,i_{k-1}}\left[F\left(x^k\right)\right|\left.
x^{k_0}\right] - F^* \leq \frac{2 \|\cP^{-1} \bar \AA^{-1} \cM \| R_0^2
}{2N + (1 - \eta)\gamma (k - k_0)}.
\label{eq:sub}
\end{align}

\item Suppose that a fixed probability distribution $\{p_i\}_{i=1}^N >
  0$ is used throughout to sample the blocks and that \eqref{eq:H2}
  holds for all $k$. {Then, defining}
  \begin{equation}
    \bar k_0 \coloneqq \left\lceil\max\left\{ 0, \frac{\log
      \frac{F\left( x^0 \right) - F^*}{\left\| \cP^{-1}
	\bar{\AA}^{-1} \cM\right\| \bar\pi
	R_0^2}}{\log\left( \frac{2}{2 - \left( 1 - \eta\right) \gamma
	\bar\pi}\right)}\right\}\right\rceil,
    \label{eq:k0bar}
  \end{equation}
  (with $\bar\pi$ defined in \eqref{eq:barpi}), {we have for
    all $k<\bar k_0$ that the expected objective satisfies}
  \begin{equation}
    \E_{i_0,\dotsc,i_{k-1}}\left[ F\left( x^k \right) - F^*
		\right] \leq \left( 1 - \frac{(1 - \eta)\gamma
      \bar\pi}{2} \right)^k \left( F\left( x^0
    \right) - F^* \right),
    \label{eq:upper1}
  \end{equation}
  {while for all $k \geq \bar k_0$, we have}
  \begin{align}
	  \E_{i_0,\dotsc,i_{k-1}}\left[F\left(x^k\right) -
		  F^* \right] \leq \frac{2 \left\|\cP^{-1} \bar \AA^{-1} \cM
		  \right\| R_0^2 }{2N + (1 - \eta)\gamma (k - \bar k_0)}.
\label{eq:upper2}
\end{align}
\end{enumerate}
\end{theorem}
\begin{proof}
{We first prove Part 1.} Consider Lemma \ref{lemma:Q}.  For the
  general convex case, we have $\mu = 0$ in {the OSSC condition} \eqref{eq:strong}, so
  \eqref{eq:Q} reduces to
\begin{align}
\label{eq:Q2}
&~\E_{i_k}\left[\left.\alpha^k_{i_k} Q^{k*}_{i_k}\right|
x^k \right]\\
\nonumber
\leq&~ \theta\lambda \left(F^* - F\left(x^k\right) \right)+
\frac{\theta^2\lambda^2}{2} \left(x^k -
P_\Omega\left(x^k\right)\right)^\top \cP_k^{-1}\AA_k^{-1} {\HH_k} \left(x^k
- P_\Omega\left(x^k\right)\right),
\end{align}
for all $\lambda \in [0,1]$ and all $\theta \in [0, \pi^k]$.
Setting $\theta =
  \pi^k$, we note that the right-hand side of \eqref{eq:Q2} is a
  strongly convex function of $\lambda$ for $x \notin \Omega$, so by
  minimizing explicitly with respect to $\lambda$, we obtain
\begin{equation}
\lambda = \min\left\{ 1, \frac{F\left( x^k \right) - F^* }{ \left(x^k
	- P_\Omega\left(x^k\right)\right)^\top \cP_k^{-1}\AA_k^{-1}
\cH_k \left(x^k - P_\Omega\left(x^k\right)\right)\pi^k} \right\}.
\label{eq:lambda}
\end{equation}
{With this setting of $\lambda$, when}
\begin{equation*}
	F\left( x^k \right) - F^* \geq \left(x^k -
	P_\Omega\left(x^k\right)\right)^\top  \cP_k^{-1}\AA_k^{-1}
	\cH_k \left(x^k - P_\Omega\left(x^k\right)\right) \pi^k,
\end{equation*}
we have $\lambda = 1$ and \eqref{eq:Q2} becomes
\begin{align}
	\E_{i_k}\left[\left.\alpha_{i_k}^k Q^{k*}_{i_k} \right|
	x^k\right] \leq \tfrac12 \pi^k \left(F^* - F\left(x^k\right)
	\right).
\label{eq:Q3}
\end{align}
By combining \eqref{eq:Q3} and \eqref{eq:usage}, we have proved
\eqref{eq:earlylinear}.

{Next, we prove Part 2.} Consider \eqref{eq:Q2} with
$\alpha^k_{i_k}$ replaced by $\bar \alpha_{i_k}$ (so that $\AA_k$ is
replaced by $\bar\AA$) and $p^k_{i_k}$ replaced by $p_{i_k}$ (so that
$\cP_k$ is replaced by $\cP$). For any $k \geq k_0$, we define
\begin{equation*}
	\delta_k \coloneqq \E_{i_{k_0},\dotsc, i_{k-1}} \left[\left. F\left( x^k
	\right) - F^*\right| x^{k_0}\right].
\end{equation*}
By applying the definition \eqref{eq:r0} and the bound \eqref{eq:H2}
on the right-hand side of the updated \eqref{eq:Q2},
and then taking expectations on both sides over $i_{k_0},\dotsc,
i_{k-1}$ conditional on $x^{k_0}$, we have that
\begin{align}
\label{eq:Q4}
\E_{i_{k_0},\dotsc, i_k}\left[\left. \bar \alpha_{i_k}
Q^{k*}_{i_k} \right| x^{k_0}\right]
\leq -\theta\lambda \delta_k +
\frac{\theta^2\lambda^2}{2}\left\|\cP^{-1} \bar \AA^{-1} \cM\right\|
R_0^2,
\end{align}
for all $\lambda \in [0,1]$ and all $\theta \in [0, \bar \pi]$.
Setting $\theta = \bar\pi$ in \eqref{eq:Q4}, we have from
\eqref{eq:k0} that since Algorithm~\ref{alg:icd} is a descent method,
$\delta_k < \bar\pi \left\| \cP^{-1} \bar \AA^{-1} \cM \right\| R_0^2$,
for all $k \ge k_0$.  Therefore, we can use
\begin{equation*}
	\lambda = \frac{\delta_k}{\bar\pi
	\left\| \cP^{-1} \bar \AA^{-1} \cM\right\| R_0^2}
\end{equation*}
in \eqref{eq:Q4} to obtain
\begin{align}
\nonumber
\E_{i_{k_0},\dotsc,i_k} \left[ \left. \bar \alpha_{i_k}
Q^{k*}_{i_k}\right| x^{k_0}\right]
\leq&~ -\bar\pi \frac{\delta_k^2
}{2 \bar\pi \left\| \cP^{-1} \bar
\AA^{-1} \cM \right\| R_0^2}\\
=&~ -\frac{\delta_k^2}{2 \left\| \cP^{-1} \bar \AA^{-1} \cM\right\|
R_0^2}.
\label{eq:newbound}
\end{align}
Therefore, by taking expectation on \eqref{eq:usage} over
$i_{k_0},\dotsc, i_k$ conditional on $x^{k_0}$, and using
\eqref{eq:newbound}, we obtain
\begin{equation}
	\delta_{k+1} \leq \delta_k - \frac{\left( 1 - \eta
	\right)\gamma\delta_k^2}{2\left\| \cP^{-1} \bar \AA^{-1} \cM\right\|
R_0^2}.
\label{eq:deltak}
\end{equation}
By dividing both sides of \eqref{eq:deltak} by $\delta_k \delta_{k+1}$
and noting from \eqref{eq:armijo} and Lemma \ref{lemma:linesearch}
that $\{F(x_k)\}$ and therefore $\{\delta_k\}$ is descending, we obtain
\begin{equation}
\frac{1}{\delta_{k}}
\leq \frac{1}{\delta_{k+1}} - \frac{\left( 1 - \eta
\right)\gamma\delta_k}{2\delta_{k+1} \left\| \cP^{-1} \bar \AA^{-1}
\cM\right\| R_0^2}
\leq \frac{1}{\delta_{k+1}} - \frac{\left( 1 - \eta
\right)\gamma}{ 2\left\| \cP^{-1} \bar \AA^{-1} \cM\right\| R_0^2}.
\label{eq:telescpoe}
\end{equation}
By summing and telescoping \eqref{eq:telescpoe}, we obtain
\begin{equation}
	\frac{1}{\delta_k} \geq \frac{1}{\delta_{k_0}} +
\left( k - k_0\right) \frac{\left( 1 - \eta
\right)\gamma }{2\left\| \cP^{-1} \bar \AA^{-1} \cM\right\| R_0^2}.
\label{eq:summed}
\end{equation}
Finally, note that because $\bar{\alpha}_i \in [0,1]$ for
$i=1,\dotsc,N$, \eqref{eq:k0} implies that
\begin{equation}
	\frac{1}{\delta_{k_0}} \geq \frac{1}{\bar\pi \left\| \cP^{-1} \bar \AA^{-1} \cM\right\| R_0^2}
	\geq
	\frac{1}{\min_i  p_i \left\| \cP^{-1} \bar \AA^{-1} \cM\right\|
R_0^2}.
\label{eq:deltak0}
\end{equation}
Next, it is straightforward that the solution to
	\[
		\min_{p_1,\dotsc,p_N} \, \frac{1}{\min_{1 \le i \le N} p_i}
\quad		\text{subject to}\;\; \sum_{i=1}^N p_i = 1,  
		\;\; p_i \geq 0, \;\; i=1,\dotsc,N
	\]
is $p_i \equiv 1 / N$, and the corresponding objective value is $N$.
Therefore, \eqref{eq:deltak0} {further implies} that
\begin{equation*}
	\frac{1}{\delta_{k_0}} \geq
	\frac{N}{\left\| \cP^{-1} \bar \AA^{-1} \cM\right\| R_0^2}.
\end{equation*}
By combining this inequality with \eqref{eq:summed}, we obtain
\eqref{eq:sub}.

{For Part 3, we again start from \eqref{eq:Q2} and
 replace $\alpha^k_{i_k}$ with $\bar{\alpha}_{i}$ in
\eqref{eq:Q2} to obtain
\begin{align}
\label{eq:P3Q1}
	&~\E_{i_k}\left[\left.\bar{\alpha}_{i_k} Q^{k*}_{i_k}\right|
x^k \right]\\
\nonumber
\leq&~ \theta\lambda \left(F^* - F\left(x^k\right) \right)+
\frac{\theta^2\lambda^2}{2} \left(x^k -
P_\Omega\left(x^k\right)\right)^\top \cP^{-1}\bar\AA^{-1} \HH_k \left(x^k
- P_\Omega\left(x^k\right)\right),
\end{align}
for all $\lambda \in [0,1]$ and all $\theta \in [0, \bar \pi]$.
By applying \eqref{eq:r0} and \eqref{eq:H2}, we have
\begin{align*}
	&~\E_{i_k}\left[\left.\bar{\alpha}_{i_k} Q^{k*}_{i_k}\right|
x^k \right]\\
\nonumber
\leq&~ \theta\lambda \left(F^* - F\left(x^k\right) \right)+
\frac{\theta^2\lambda^2}{2} \left\|x^k -
P_\Omega\left(x^k\right)\right\| \|\cP^{-1}\bar\AA^{-1}
\cM\| \left\|x^k
- P_\Omega\left(x^k\right)\right\|\\
\leq&~ \theta\lambda \left(F^* - F\left(x^k\right) \right)+
\frac{\theta^2\lambda^2}{2} \|\cP^{-1}\bar\AA^{-1}
\cM\|R_0^2.
\end{align*}
Now we take expectation over $i_0,\dotsc,i_{k-1}$ on both sides of
this inequality (noting that the last term on the right-hand side are
all constants that do not depend on $i_k$) to obtain
\[
\E_{i_0,\dotsc,i_k}\left[\bar \alpha_{i_k}
Q^{k*}_{i_k}\right]
\leq -\theta\lambda \E_{i_0,\dotsc,i_{k-1}} \left[ F^* - F\left( x^k
\right) \right] +
\frac{\theta^2\lambda^2}{2}\left\|\cP^{-1} \bar \AA^{-1} \cM\right\|
R_0^2.
\]
By defining
\begin{equation*}
	\hat \delta_k \coloneqq \E_{i_{0},\dotsc, i_{k-1}} \left[ F\left( x^k
	\right) - F^*\right]
\end{equation*}
and setting $\theta = \bar{\pi}$, we have that
\begin{align}
\E_{i_0,\dotsc,i_k}\left[\bar \alpha_{i_k}
Q^{k*}_{i_k}\right]
\leq -\bar\pi\lambda \hat \delta_k +
\frac{\bar \pi^2\lambda^2}{2}\left\|\cP^{-1} \bar \AA^{-1} \cM\right\|
R_0^2.
\label{eq:P3Q2}
\end{align}
The minimum of the right-hand side happens when
\[
	\lambda = \min\left\{ 1, \frac{\hat \delta_k }{\bar{\pi}\|\cP^{-1} \bar
	\AA^{-1} \cM\|R_0^2}\right\}.
\]
When the expected function value satisfies
\begin{equation}
	\hat \delta_k  \geq \bar{\pi}\|\cP^{-1} \bar \AA^{-1} \cM\|R_0^2,
\label{eq:turning}
\end{equation}
the minimizer is $\lambda = 1$, and the bound becomes
\begin{align*}
\E_{i_0,\dotsc,i_k}\left[\bar \alpha_{i_k}
Q^{k*}_{i_k}\right]
\leq -\bar\pi\hat \delta_k +
\frac{\bar \pi^2}{2}\left\|\cP^{-1} \bar \AA^{-1} \cM\right\|
R_0^2
\leq -\frac{\bar \pi \hat \delta_k}{2}.
\end{align*}
Now we consider \eqref{eq:usage} and take expectation over $i_0,\dotsc,
i_k$ on both sides. Note that $Q^{k*}_{i_k} \leq 0$ so the upper bound is still
valid if we replace $\alpha^k_{i_k}$ with $\bar \alpha_{i_k}$.
Thus we obtain
\begin{align} \label{eq:P3Q3}
	\hat \delta_{k+1} - \hat{\delta}_{k} =
        \E_{i_0,\dotsc,i_k}\left[F\left( x^{k+1} \right) - F\left( x^k
          \right)\right] & \leq (1 - \eta) \gamma \E_{i_0,\dotsc,i_k}
        \left[ \bar \alpha_{i_k} Q^{k*}_{i_k}\right] \\
        \nonumber
        & \leq -\frac{ (1 -
          \eta) \gamma\bar \pi \hat \delta_k}{2}.
\end{align}
By rearranging the inequality above, we get the linear convergence of
\[
	\hat \delta_{k+1} \leq \hat{\delta}_{k} \left(1 -\frac{(1 - \eta)
		\gamma \bar \pi}{2}\right).
\]
Therefore, we always get the bound
\[
	\hat \delta_k \leq \left( 1 - \frac{(1 - \eta) \gamma \bar \pi
	}{2}  \right)^k \left( F(x^0) - F^* \right)
\]
until $\hat \delta_{k-1} <\bar{\pi}\|\cP^{-1} \bar \AA^{-1}
\cM\|R_0^2$.
Note that $\bar k_0$ is obtained as the first value of $k$ such that
\[
	\left( 1 - \frac{(1 - \eta) \gamma \bar{\pi}}{2} \right)^k \left(
	F\left( x^0 \right) - F^* \right) \leq  \bar{\pi}\|\cP^{-1} \bar
\AA^{-1} \cM\|R_0^2.
\]
Therefore, for $k < \bar k_0$, the upper bound in \eqref{eq:upper1} is
larger than $\bar{\pi}\|\cP^{-1} \bar \AA^{-1} \cM\|R_0^2$, so the
rate \eqref{eq:upper1} remains valid.
Note that if $\hat \delta_k \leq  \bar{\pi}\|\cP^{-1} \bar
\AA^{-1} \cM\|R_0^2$ has already held true for some $k < \bar k_0$,
clearly this bound is still valid.
On the other hand,
after $\bar k_0$, we are guaranteed that \eqref{eq:turning} must stop
holding.
Thus the minimizer for \eqref{eq:P3Q2} becomes $\lambda = \hat
\delta_k / (\bar{\pi}\|\cP^{-1} \bar \AA^{-1} \cM\|R_0^2)$.
We then start from the first inequality of \eqref{eq:P3Q3}
and get
\[
	\hat{\delta}_{k+1} - \hat{\delta}_k \leq -\frac{(1 - \eta) \gamma
		\hat{\delta}_k^2}{2 \|\cP^{-1} \bar \AA^{-1} \cM\|R_0^2}.
\]
Following the same derivation we had in Part 2, we can get
\[
	\frac{1}{ \hat{\delta}_k} \leq \frac{1}{\hat{\delta}_{k+1}} -
	\frac{(1 - \eta) \gamma }{2 \|\cP^{-1} \bar \AA^{-1} \cM\|R_0^2}.
\]
By summing and telescoping the result above, we get
\[
	\frac{1}{ \hat{\delta}_k} \geq \frac{1}{\hat{\delta}_{\bar k_0}} + (k -
	\bar k_0) \frac{(1 - \eta) \gamma }{2 \|\cP^{-1} \bar \AA^{-1}
\cM\|R_0^2}.
\]
Following the same argument in Part 2, we get the final claim in Part
3.}
\qed
\end{proof}
The rate indicated by Part 1 of Theorem~\ref{thm:sub} has been
observed frequently in practice, and some restricted special cases
without a regularizer have been discussed in the literature
\cite{LeeW16a,WriL16b}. {To our knowledge, ours is the first
  theoretical result for BCD-type methods on general regularized
  problems \eqref{eq:f}. The global convergence bounds in other works
  depend on $R_0^2 + F(x^0) - F^*$, whereas our results} significantly
weaken the dependence on the initial objective value.

We can see from Part 2 of Theorem~\ref{thm:sub} that the optimal
probability distribution after $k_0$ iterations is {the one
  for which $\| \cP^{-1} \bar\cA^{-1} \cM \|$ is minimized, that is,}
\begin{equation}
p_i = \frac{M_i \bar \alpha_i^{-1}}{\sum_{j} M_j \bar
\alpha_j^{-1}}.
\label{eq:pi}
\end{equation}
It is possible to replace $\bar \alpha_i$ and $M_i$ with the values
$\alpha^k_i$ and $\|H^k_i\|$ (respectively), to obtain adaptive
probabilities and possibly sharper rates, but we fix the probabilities
for the sake of more succinct analysis. {We discuss in
  Section~\ref{sec:related} some issues relating to the use of
  adaptive probabilities.} 

We now consider the case that $F$ satisfies the quadratic
growth condition
\begin{equation}
	F(x) - F^* \geq \frac{\mu}{2}\left\|x - P_{\Omega}\left( x
	\right) \right\|^2
	\label{eq:growth}
\end{equation}
for some $\mu > 0$.  {This condition is implied by the OSSC
  condition} \eqref{eq:strong} but not vice versa.  The following
theorem shows a global Q-linear convergence result for this case.
\begin{theorem}
\label{thm:linear0}
Assume that $f$ and $\psi$ are convex and that \eqref{eq:compL} and
\eqref{eq:growth} hold for some $L_1,\dotsc,L_N, \mu > 0$. Suppose
that at the $k$th iteration of Algorithm~\ref{alg:icd},
\eqref{eq:approx} is satisfied with some $\eta \in [0,1)$ and $H_i^k$
  is chosen such that \eqref{eq:H} holds for some $m_i > 0$ and all
  $i=1,2,\dotsc,N$, so that the step sizes $\alpha^k_i$ are all
  bounded away from $0$, as indicated by Lemma~\ref{lemma:linesearch}.
  Then given any probability distribution $\{p^k_i\} > 0$, with
  $\pi^k$ defined as in \eqref{eq:pik}, we have that the expected
  decrease at iteration $k$ is {
  \begin{equation}
\label{eq:Qlinear.0}
    \frac{\E_{i_k}\left[ F\left(x^{k+1}\right) - F^*\right|\left.
        x^k\right]} {F\left(x^k \right) - F^*} \le 1-(1-\eta)\gamma \rho_k,
  \end{equation}
  where $\rho_k$ is bounded below}
 by the following quantities:
  \begin{subequations} \label{eq:Qlinear0.cases}
	  {
    \begin{align}
      \label{eq:Qlinear0.2}
     \frac{\mu}{4 \|\cP^{-1}_k \AA^{-1}_k \cH_k \|}, &\quad
 	 \mbox{\rm if } \quad \frac{\mu}{ 2 \|\cP^{-1}_k \AA^{-1}_k \cH_k \|
 	\pi^k}\le 1, \\
     \label{eq:Qlinear0.3}
     \pi^k  \left( 1 -
 	\frac{\pi^k \|\cP^{-1}_k
 	\AA^{-1}_k \cH_k \|}{\mu} \right), &\quad \mbox{\rm otherwise}.
      \end{align}
  }
    \end{subequations}
\end{theorem}
\begin{proof}
By \eqref{eq:usage}, \eqref{eq:Q2}, the Cauchy-Schwarz inequality, and
\eqref{eq:growth}, we have
\begin{align}
\nonumber
&~\E_{i_k}\left[ F\left( x^{k+1} \right) - F\left( x^k \right) \mid
x^k \right]\\
\nonumber
\leq &~\gamma (1 - \eta) \E_{i_k}\left[\alpha_{i_k}^k
Q^{k*}_{i_k}\right]\\
\label{eq:Q5}
\leq&~\gamma (1 - \eta)\theta\left(F\left(x^k\right) -
F^*\right)\left( -\lambda  +
\frac{\theta\lambda^2\left\|\cP^{-1}_k \AA^{-1}_k
  \cH_k \right\|}{\mu}\right),
\end{align}
for all $\lambda \in [0,1]$ and all $\theta \in [0, \pi^k]$.  By the
same argument as in the previous proofs, we let $\theta = \pi^k$.  By
minimizing the right-hand side of \eqref{eq:Q5} over $\lambda \in
[0,1]$, we obtain the two cases \eqref{eq:Qlinear0.2} and
\eqref{eq:Qlinear0.3}.
\qed
\end{proof}

{We can improve on Theorem~\ref{thm:linear0} for problems
  satisfying the OSSC condition \eqref{eq:strong}.}
\begin{theorem}
\label{thm:linear}
Assume that $f$ and $\psi$ are convex and that \eqref{eq:compL} and
\eqref{eq:strong} hold for some $L_1,\dotsc,L_N, \mu > 0$.
Suppose that at the $k$th iteration of Algorithm~\ref{alg:icd},
\eqref{eq:approx} is satisfied for some $\eta \in [0,1)$ and $H_i^k$
  is chosen such that \eqref{eq:H} holds for some $m_i > 0$ and all
  $i=1,2,\dotsc,N$ so that the step sizes $\alpha^k_i$ are all lower
  bounded away from $0$ as indicated by Lemma~\ref{lemma:linesearch}.
  Then given any probability distribution $\{p^k_i\} > 0$, and with
  $\pi^k$ defined as in \eqref{eq:pik}, {the expected function decrease
  at iteration $k$ is the same as \eqref{eq:Qlinear.0}, but with
  $\rho_k$ lower-bounded by
    \begin{equation} \label{eq:Qlinear.cases}
       \rho_k \ge \left(\frac{1}{
	\pi^k} + \max_i \frac{\left\|H^k_i\right\|}{\mu
	\alpha_i^k p_i^k}\right)^{-1}.
\end{equation}
}
\end{theorem}
\begin{proof}
  {We note by bounding the last term in} \eqref{eq:Q} that
  \begin{align*}
    \E_i\left[\left.\alpha^k_i Q^{k*}_i \right| x^k\right]  &\le
    \theta\lambda \left(F^* - F\left(x^k\right) \right)- \tfrac12 \mu \theta
	\lambda \left( 1 - \lambda \right) \left\|x^k -
P_\Omega\left(x^k\right)\right\|^2 +\\
& \tfrac12 \theta^2\lambda^2 \left\| x^k - P_{\Omega}( x^k ) \right\|^2
    \| \cP_k^{-1}\AA_k^{-1} \HH_k\|.
  \end{align*}
  Thus by setting $\lambda = \mu/(\mu + \|\cP_k^{-1}\AA_k^{-1}
  \cH_k\|\theta) \in [0,1]$, the last two terms cancel.  Then by setting
  $\theta = \pi^k$, we obtain
\begin{align}
  \nonumber
  \E_{i_k}{\left[\left.\alpha_{i_k}
      Q^{k*}_{i_k} \right|x^k \right]}
  &\leq
  \frac{\mu\theta}{\mu + \left\|\cP^{-1}_k \AA^{-1}_k
    \cH_k \right\|\theta}\left( F^* - F\left( x^k
  \right) \right)\\
  \nonumber
&= \frac{1}{\frac{1}{\theta} + \frac{\left\|\cP^{-1}_k \AA^{-1}_k
      \cH_k \right\|}{\mu}}\left( F^* - F\left( x^k \right) \right)\\
  &= \frac{1}{\frac{1}{\pi^k} + \max_i
    \frac{\left\|H^k_i\right\|}{\mu \alpha^k_i p^k_i}}\left( F^* -
  F\left( x^k \right) \right).
  \label{eq:linearfinal}
\end{align}
By combining \eqref{eq:linearfinal} and \eqref{eq:usage}, we obtain
the desired result.   \qed
\end{proof}

For problems on which Theorem~\ref{thm:linear0} or \ref{thm:linear}
holds,
Theorem~\ref{thm:sub} is also applicable, and the early linear
convergence rate {can be faster than the global rates
described in Theorems~\ref{thm:linear0} and \ref{thm:linear} (always
better than the rate in Theorem~\ref{thm:linear0} and for
Theorem~\ref{thm:linear} it depends on the value of $\mu$ and
$\|H^k_i\|$).}  Thus,
we could sharpen the global iteration complexity for problems
satisfying {the OSSC condition} \eqref{eq:strong} with $\mu >
0$ by {using Theorem~\ref{thm:sub}}.
We also notice that the rate in Theorem~\ref{thm:linear} is faster
than that in Theorem~\ref{thm:linear0}, which is why we consider these
two conditions separately.

Note too that with knowledge of $\alpha_i$ and $\|H^k_i\|$, we could
in principle minimize the expected gap $\E_{i_k}\left[
F\left(x^{k+1}\right) - F^*\right|\left. x^k\right]$ by minimizing the
denominator on the right-hand side of \eqref{eq:Qlinear.cases} and
\eqref{eq:Qlinear0.cases} with respect to $p^k_i$ over $p^k_i>0$ and
$\sum_i p_i^k=1$. Such an approach is not practical  except
in the special cases discussed in Section~\ref{sec:rcd}, as it is
unclear how to find $\alpha_i^k$ and $\|H^k_i\|$ in general for the
blocks not selected.

Theorems~\ref{thm:linear0} and \ref{thm:linear} suggest that larger step
sizes lead to faster convergence.
  When
  $H_i^k$ incorporates curvature information of $f$,
  {empirically} we tend to have much larger step sizes than
  the lower bound predicted in Lemma~\ref{lemma:linesearch},
  and thus the practical performance of using the
Hessian or its approximation usually outperforms using a multiple of
the identity as $H_i^k$.

All the results here can be combined in a standard way with Markov's
inequality to get high-probability bounds for the objective value.  We
omit these results.

\subsection{Nonconvex Case}
\label{subsec:noncvx}

When $f$ is not necessarily convex, we cannot use Lemma~\ref{lemma:Q}
to estimate the expected model decrease at each iteration, and we
cannot guarantee convergence to the global optima.  Instead, we
analyze the convergence of certain measures of stationarity.

The first measure we consider is how fast the optimal objective of the
subproblem \eqref{eq:quadratic} converges to zero.  Since the
subproblems are strongly convex, this measure is zero if and only if
the optimal solution is the zero vector, implying {that the
  algorithm will not step away from this point.} These claims are
verified in the following lemma.
\begin{lemma}
\label{lemma:opt}
{At iteration $k$, assume that} $H_i^k \succeq m_i I$ for
some $m_i > 0$, $i=1,2,\dotsc,N$ in
\eqref{eq:quadratic}-\eqref{eq:Qdef}, and that \eqref{eq:compL} is
satisfied for some positive values $L_1,L_2,\dotsc,L_N$. Then for any
positive step sizes $\{\alpha_i^k\}_{i=1}^N > 0$ and any probability
distribution $\{p_i^k\}_{i=1}^N > 0$, we have
\begin{equation}
\E_i[\alpha_i^k Q_i^{k*}] = 0\; \Leftrightarrow \; Q_i^{k*} = 0, \;
i=1,\dotsc,N \; \Leftrightarrow \; 0 \in \partial F\left( x^k \right),
	\label{eq:opt}
\end{equation}
where $\partial F(x^k) = \nabla f(x^k) + \partial \psi(x^k)$ is the
generalized gradient of $F$ at $x^k$.
\end{lemma}
\begin{proof}
From \eqref{eq:deltabound2} in Lemma~\ref{lemma:linesearch}, by
setting $\eta = 0$ we see that for all $i$ and $k$ we have $Q^{k*}_i
\leq 0$, proving the first equivalence in \eqref{eq:opt}.  To prove
the second equivalence, we first notice that since $Q_i^k$ are all
strongly convex and $Q_i^k(0) \equiv 0$, $Q_i^{k*} = 0$ if and
only if $d^{k*}_i = 0$, where $d^{k*}_i$ is defined in
\eqref{eq:quadratic}.  Therefore, it suffices to prove that
\begin{equation}
	d^{k*}_i = 0 \quad \Leftrightarrow \quad -\nabla_i f\left( x^k
        \right) \in \partial \psi_i\left( x^k_i \right), \;\; i=1,\dotsc,N.
	\label{eq:toprove}
\end{equation}
{From optimality of \eqref{eq:quadratic}, we have}
\begin{equation}
- \left(\nabla_i f\left( x^k \right) + H_i^k d^{k*}_i \right) \in \partial
\psi_i \left(x^k_i + d^{k*}_i \right).
\label{eq:subdiff}
\end{equation}
When $d^{k*}_i = 0$, \eqref{eq:subdiff} implies that $-\nabla_i
f(x^k) \in \partial \psi_i(x^k_i)$.
Conversely, if
\begin{equation}
-\nabla_i f(x^k) \in \partial \psi_i (x_i^k).
\label{eq:optcond}
\end{equation}
We have from the convexity of $\psi_i$ { together with
\eqref{eq:optcond} and \eqref{eq:subdiff} that 
\begin{align*}
		\psi_i\left(x_i^k + d_i^{k*}\right) &\geq
		\psi_i\left(x^k_i\right) -\nabla_if(x^k)^\top d^{k*}_i,\\
		\psi_i\left(x_i^k\right) &\geq \psi_i\left(x^k_i +
		d^{k*}_i\right) -(d^{k*}_i)^\top(-\nabla_i f(x^k) - H^k_i d^{k*}_i).
\end{align*}
}
By adding these two inequalities, we obtain
$0 \geq \left(d^{k*}_i\right)^\top  H^k_i d^{k*}_i$, so that
$d^{k*}_i=0$ by the positive definiteness of $H^k_i$.
\qed
\end{proof}

The second measure of convergence is the following:
\begin{equation}
	G_k \coloneqq \arg\min_d\, \nabla f(x^k)^\top  d +
	\tfrac12 d^\top  d + \psi(x^k + d).
	\label{eq:Gk}
\end{equation}
From Lemma \ref{lemma:opt}, it is clear that $G_k = 0$ if and only if
$0 \in \partial F(x^k)$, so $G_k$ can serve as an indicator for
closeness to stationarity.

We show convergence rates for the two measures proposed above.
\begin{theorem}
\label{thm:Qbound}
{Given any $x^0$ in Algorithm~\ref{alg:icd}},  let
$\{\alpha^k_i\}_{i=1}^N > 0$ be the step sizes generated by the line
search procedure for $k=0,1,2,\dotsc$.  If $H_i^k \succeq 0$ for all
$i$ and $k$, we have
\begin{equation}
  \min_{0 \leq k \leq T}\, \left|
  \E_{i_0,\dotsc,i_k} \left[  \alpha_{i_k}^k Q_{i_k}^k\left(
    d^k_{i_k} \right) \right] \right| \leq \frac{F\left( x^0
    \right) -F^*}{\gamma\left(T+1\right)}, \quad \mbox{for all $T \geq 0$}.
\label{eq:Qopt}
\end{equation}
Moreover,
{$\E_{i_0,\dotsc,i_k}\left[ \alpha_{i_k}^k
    Q_{i_k}^k(d^k_{i_k}) \right] \rightarrow 0$} as $k$ approaches
infinity.
\end{theorem}
\begin{proof}
Taking expectation on \eqref{eq:usage} over $i_k$, we obtain
\begin{equation}
	\E_{i_k}\left[ \left. F\left( x^{k+1} \right)\right| x^k \right] -
	F\left( x^k \right)
	\leq \gamma \E_{i_k}\left[ \left. \alpha_{i_k} Q^k_{i_k}\left( d^k_{i_k}
	\right)\right| x^k \right].
\label{eq:expecteddecrease}
\end{equation}
By taking expectation on \eqref{eq:expecteddecrease} over
$i_0,\dotsc,i_{k-1}$ and summing over $k=0,1,\dotsc, T$, and noting
from \eqref{eq:approx} and Lemma~\ref{lemma:linesearch} that
$Q_i^k(d_i^k) \leq 0$ for all $k$ and all $i$, we obtain
\begin{align}
\nonumber
&~\gamma \sum_{k=0}^T \left|\E_{i_0,\dotsc,i_k} \left[ \alpha_{i_k}
Q^k_{i_k}\left( d^k_{i_k} \right) \right]\right|\\
\nonumber
=&~ -\gamma \sum_{k=0}^T  \E_{i_0,\dotsc,i_k} \left[
\alpha_{i_k} Q^k_{i_k}\left( d^k_{i_k} \right)  \right]\\
\nonumber
\leq
&~\sum_{k=0}^T \left\{ \E_{i_0,\dotsc,i_{k-1}} \left[  F\left( x^k
\right) \right] - \E_{i_0,\dotsc,i_k} \left[  F\left(
x^{k+1} \right)  \right] \right\}\\
=&~ F\left( x^0 \right) - \E_{i_0,\dotsc,i_T} \left[ F\left(
x^{T+1} \right) \right]
\leq~ F\left( x^0 \right) - F^*.
\label{eq:summable}
\end{align}
The result now follows from
\begin{equation*}
\sum_{k=0}^T \left|\E_{i_0,\dotsc,i_k} \left[ \alpha_{i_k}
Q^k_{i_k}\left( d^k_{i_k} \right) \right]\right|\\
\geq  \left( T+1 \right)\min_{0 \leq k \leq T}
\left|\E_{i_0,\dotsc,i_k} \left[ \alpha_{i_k}
Q^k_{i_k}\left( d^k_{i_k} \right) \right]\right|\\
\end{equation*}
The result that $\left|\E_{i_0,\dotsc,i_k} \left[ \alpha_{i_k}
  Q^k_{i_k}\left( d^k_{i_k} \right) \right]\right| \rightarrow 0$
follows from the summability implied by \eqref{eq:summable}.  \qed
\end{proof}

Unlike previous results, {the convergence speed for the right-hand side
	of \eqref{eq:Qopt}} is independent of how
accurately the subproblem is solved, the probability distributions for
sampling the blocks, and the step sizes. We next consider the second measure
\eqref{eq:Gk} and show that its {convergence behavior depends
  on these factors}.  We need the following lemma from \cite{TseY09a}.
\begin{lemma}[{\cite[Lemma~3]{TseY09a}}]
\label{lemma:Gk}
Given $x^k$, assume that $H^k_i$ satisfies \eqref{eq:H2} for some $M_i
\geq m_i > 0$ for all $i$.
Then we have
\begin{equation*}
\left\|U_i^\top  G_k\right\| \leq \frac{1 + \frac{1}{m_i} + \sqrt{1 - 2
	\frac{1}{M_i} + \frac{1}{m_i^2}}}{2}
	M_i \left\|d_i^{k*} \right\|.
\end{equation*}
\end{lemma}

By combining this lemma with Theorem~\ref{thm:Qbound}, {we
  can show a convergence rate for $\min_{0 \le k \le T}
  \E_{i_0,\dotsc, i_k} \, \|G_k\|$.}
\begin{corollary}
 Assume
that $H^k_i$ satisfies \eqref{eq:H2} for all $k=0,1,\dotsc$ and all
$i=1,2,\dotsc,N$.
Let $\{\alpha^k_i\}_{i=1}^N > 0$ be the step sizes generated by the
line search procedure. Then we have
\begin{align}
  \nonumber
& \min_{0 \leq k \leq T
}\E_{i_0,\dotsc,i_{k-1}}\left[ \left\|G_k\right\|^2 \right]\\
\leq &~\frac{F\left( x^0 \right) - F^*}{2 (1 - \eta)\gamma (T+1)}
\max_{0 \le k \le T,\; 1 \le i \le N}
\frac{M_i^2\left(1 + \frac{1}{m_i} +
\sqrt{1 - 2 \frac{1}{M_i} + \frac{1}{m_i^2}}\right)^2}{p_i^k \alpha_{i}^{k}m_i}.
\label{eq:Gkbound}
\end{align}
\label{cor:Gk}
\end{corollary}
\begin{proof}
We consider Theorem~\ref{thm:Qbound} and let $\bar k$ be the iteration that
achieves the minimum on the left-hand side of \eqref{eq:Qopt}.  We
have from \eqref{eq:approx} and Theorem~\ref{thm:Qbound} that
\begin{equation}
\frac{F\left( x^0 \right) -F^*}{\gamma\left(T+1\right)}
\geq
\left|\E_{i_0,\dotsc,i_{\bar k}} \left[\alpha_{i_{\bar k}}^{\bar
k} Q_{i_{\bar k}}^{\bar k}\left( d^{\bar k}_{i_{\bar k}}
\right) \right]\right|
\geq
-(1 - \eta)\E_{i_0,\dotsc,i_{\bar k}}\left[ \alpha_{i_{\bar k}}^{\bar
k} Q_{i_{\bar k}}^{\bar{k}*} \right].
\label{eq:intermediate4}
\end{equation}
Since $H^k_i \succeq m_i I$ from \eqref{eq:H2} and the $\psi_i$ are
convex, we have that for all $i$ and $k$, the functions $Q_i^k$ are
$m_i$-strongly convex and {hence satisfy \eqref{eq:growth}
  with $\mu = m_i$.  Therefore, we have}
\begin{equation}
	{Q_i^k(0)-Q_i^{k*}  =} -Q_i^{k*} \geq \frac{m_i}{2} \left\|
	d_i^{k*}\right\|^2, \quad \mbox{for all $k$ and all $i=1,2,\dotsc,N$.}
	\label{eq:Qd}
\end{equation}
{By substituting \eqref{eq:Qd} into \eqref{eq:intermediate4} and using
  Lemma \ref{lemma:Gk}}, we obtain
\begin{align}
\nonumber
&~\frac{F\left( x^0 \right) -F^*}{(1 - \eta)\gamma\left(T+1\right)}\\
\geq
&~ \frac12 \sum_{i=1}^N  p^{\bar k}_i \alpha_{i}^{\bar k} m_i
\E_{i_0,\dotsc,i_{\bar{k} -
1}}\left[\left\|d_i^{\bar{k}*}\right\|^2 \right]\\
\nonumber
\geq&~ 2 \sum_{i=1}^N \frac{ p_i^{\bar k} \alpha_{i}^{\bar
k}m_i}{M_i^2 \left(1 + \frac{1}{m_i} + \sqrt{1 - 2
	\frac{1}{M_i} + \frac{1}{m_i^2}}\right)^2} \E_{i_0,\dotsc,i_{\bar
	k - 1}} \left[\left\|U_i^\top  G_{\bar k} \right\|^2 \right]\\
\label{eq:Gktmp}
\geq&~ 2 \E_{i_0,\dotsc,i_{\bar k - 1}} \left[
\left\| G_{\bar k} \right\|^2 \right] \min_{1 \le i \le N}
\frac{p_i^{\bar k}\alpha_{i}^{\bar k}m_i}{M_i^2\left(1 + \frac{1}{m_i} +
\sqrt{1 - 2 \frac{1}{M_i} + \frac{1}{m_i^2}}\right)^2},
\end{align}
where in \eqref{eq:Gktmp}, we {used} the fact that $\|x\|^2 =
\sum_{i=1}^N\|U_i^\top x\|^2$ for any $x \in \R^n$.  The result
\eqref{eq:Gkbound} is {then} proved by noting that
\begin{equation*}
	\min_{0 \le k \le T}\E_{i_0,\dotsc,i_{k-1}} \,
	\|G_k\|^2  \leq \E_{i_0,\dotsc,i_{\bar
	k-1}} \, \|G_{\bar k}\|^2.
	\tag*{\qed}
\end{equation*}
\end{proof}

Corollary~\ref{cor:Gk} reveals that line search can help improve the
convergence speed as larger values of $\alpha_i^k$ make the right-hand
side of \eqref{eq:Gkbound} smaller, and non-uniform sampling can possibly lead to
faster convergence.

\section{{Randomized Block Coordinate Descent}}
\label{sec:rcd}

\begin{algorithm}[tb]
\caption{Inexact Randomized BCD with Unit Step Size for \eqref{eq:f}}
\label{alg:rcd1}
\begin{algorithmic}[1]
\STATE Given $\eta \in [0,1)$ and $x^0 \in \R^n$;
\FOR{$k=0,1,2,\dotsc$}
	\STATE Pick a probability distribution $p_1^k,\dotsc,p_N^k > 0$,
		$\sum_{i} p_i^k = 1$, and sample $i_k$ accordingly;
		\STATE Compute $\nabla_{i_k} f(x^k)$ and let $H^k_{i_k} = L_{i_k} I$;
	\STATE Approximately solve \eqref{eq:quadratic} to obtain a
	solution $d_{i_k}^k$ satisfying \eqref{eq:approx};
	\STATE $x^{k+1} \leftarrow x^k + U_{i_k} d_{i_k}^k$;
\ENDFOR
\end{algorithmic}
\end{algorithm}

\begin{algorithm}[tb]
\caption{Inexact Randomized BCD with Short Step Size for \eqref{eq:f}}
\label{alg:rcd2}
\begin{algorithmic}[1]
\STATE Given $\eta \in [0,1)$ and $x^0 \in \R^n$;
\FOR{$k=0,1,2,\dotsc$}
	\STATE Pick a probability distribution $p_1^k,\dotsc,p_N^k > 0$,
	$\sum_{i} p^k_i = 1$, and sample ${i_k}$ accordingly;
	\STATE Compute $\nabla_{i_k} f(x^k)$ and let $H^k_{i_k} = {\Lmin} I$;
	\STATE Approximately solve \eqref{eq:quadratic} to obtain a
	solution $d_{i_k}^k$ satisfying \eqref{eq:approx};
	\STATE $x^{k+1} \leftarrow x^k + \frac{{\Lmin}}{L_{i_k}} U_{i_k}
	d_{i_k}^k$;
\ENDFOR
\end{algorithmic}
\end{algorithm}

{Non-uniform sampling in coordinate descent for smooth convex
  objectives was discussed in \cite{Nes12a}. In this section, we
  extend these results to the regularized objective function
  \eqref{eq:f}}, using results from Section~\ref{sec:analysis}.  In
the non-regularized case, the update for the $i$th block described in
\cite{Nes12a} is $-\nabla_i f(x)/L_i$, which can be viewed as either
the solution of
\[
	\min_{d_i}\quad \nabla_i f(x)^\top  d_i + \tfrac12 L_i d_i^\top  d_i
\]
with unit step size, or equivalently as the solution of
\[
	\min_{d_i}\quad \nabla_i f(x)^\top  d_i + \tfrac12 {\Lmin} d_i^\top  d_i
\]
with step size {$\Lmin/L_i$} {(so that the step size is no
larger than $1$)}. {As in \cite{Nes12a}, we do not
  consider backtracking, but assume that $L_i$ is available, and thus
  an appropriate choice for $\alpha_i$ can be made.}  {When these step
  calculations are adapted to the regularized case \eqref{eq:f}, as in
  \eqref{eq:quadratic}-\eqref{eq:Qdef} with $H^k_i = L_i I$ and
  {$H^k_i=\Lmin I$}, respectively, they lose their equivalence to each other
  and give different directions. The resulting special cases of
  Algorithm~\ref{alg:icd} are shown as Algorithms~\ref{alg:rcd1} and
  \ref{alg:rcd2}.}

{We show in the following result that both approaches achieve
  a guaranteed decrease in the objective.}
\begin{lemma}
\label{lemma:line2}
Assume that \eqref{eq:compL} holds, {and consider iteration
  $k$ of Algorithm~\ref{alg:icd}.  If the $i$th block is selected for
updating}, and $H_i^k \succeq c_i I$ in
\eqref{eq:Qdef} for some $c_i \in (0,L_i]$, then $\hat \alpha_i
\coloneqq c_i/L_i$ satisfies
\begin{equation}
F(x^k + \alpha U_i d^k_i) - F(x^k) \leq \alpha
Q^k_{i}(d_i^k),\, \mbox{for all $d_i^k \in
\R^{n_i}$ and all $\alpha \in [0,\hat \alpha_i]$}.
\label{eq:1L}
\end{equation}
\end{lemma}
\begin{proof}
Because $c_i \in (0, L_i]$, we have $\hat \alpha_i = c_i / L_i \in
(0,1]$.
Thus from \eqref{eq:compL} and the convexity of $\psi$, we have
for any $\alpha \in [0, \hat \alpha_i]$ that
\begin{align*}
&~F\left( x^k+ \alpha U_i d^k_i \right)\\
=&~ f \left( x^k + \alpha U_i d^k_i \right) + \psi \left( x^k + \alpha U_i
d^k_i \right)\\
\leq&~ f \left( x^k \right) + \alpha \nabla_i f\left( x^k \right)^\top  d^k_i +
	\tfrac12 L_i\alpha^2 \left\|d^k_i\right\|^2 + \alpha \psi \left( x^k + U_i
	d^k_i \right) + \left( 1 - \alpha \right) \psi \left( x^k \right)\\
=&~ f\left( x^k \right) + \psi\left( x^k \right) + \alpha \left[
\nabla_i f(x^k)^\top d^k_i +\tfrac12 L_i\alpha \left\|d^k_i\right\|^2 + \psi
\left( x^k + U_i d^k_i \right) - \psi\left( x^k \right) \right]\\
\leq&~ F\left( x^k \right) + \alpha Q^k_{i}\left( d^k_i \right).
\end{align*}
In the last inequality, we used the fact that for the term $H_i$
appearing in $Q^k_i(d^k_i)$, we have
\[
	H_i \succeq c_i I = \hat \alpha_i L_i I \succeq \alpha L_i I.
\tag*{\qed}
\]
\end{proof}

With the help of Lemma~\ref{lemma:line2}, we can discuss the iteration
complexities of {randomized BCD (Algorithms~\ref{alg:rcd1}
  and \ref{alg:rcd2})} with different sampling strategies.  We first
consider the interpretation in Algorithm~\ref{alg:rcd1}, starting from
the case in which $f$ is convex. The results below are direct
applications of Theorem~\ref{thm:sub}.
\begin{corollary}
\label{cor:rcd1}
Consider Algorithm \ref{alg:rcd1} {applied to \eqref{eq:f}
  with convex $f$}, and assume that \eqref{eq:compL} holds.
 The expected objective value satisfies the
  following.
\begin{enumerate}
\item {With uniform sampling $p_i^k \equiv 1 / N$, we have
  the following.}
\begin{enumerate}
\item If $F\left( x^k \right) - F^* \geq \left( x^k - P_{\Omega}\left(
  x^k \right) \right)^\top \cL \left( x^k - P_{\Omega}\left( x^k
  \right) \right)$, where
\begin{equation}
\cL \coloneqq \diag(L_1 I_{n_1},\dotsc, L_N I_{n_N}),
\label{eq:L}
\end{equation}
we have
\begin{equation*}
\E_{i_k}\left[ F\left( x^{k+1} \right) - F^*\mid x^k \right] \leq
\left( 1 - \frac{\left( 1 - \eta \right)}{2N} \right)\left( F\left(
x^k \right) - F^* \right).
\end{equation*}
\item For all $k \geq k_0$, where $k_0 \coloneqq
\arg\min\{k: F\left( x^k \right) - F^* < \Lmax
R_0^2\}$, we have
\begin{equation*}
\E_{i_{k_0},\dotsc,i_{k-1}}\left[ F\left( x^k \right) \mid x^{k_0}
  \right] - F^* \leq \frac{2N\Lmax R_0^2}{2N + (1 - \eta) (k - k_0) }.
\end{equation*}
\end{enumerate}
\item When $p_i^k$ are defined as
\begin{equation}
p_i^k = \frac{ L_i }{ N\Lavg}, \quad i = 1,2,\dotsc,N,
\label{eq:lip}
\end{equation}
we have the following.
\begin{enumerate}
\item If $ F\left( x^k \right) - F^* \geq \Lmin \left\| x^k -
  P_{\Omega}\left( x^k \right) \right\|^2$, then
\[
\E_{i_k}\left[ F\left( x^{k+1} \right) - F^*\mid x^k
\right] \leq \left( 1 - \frac{\Lmin\left( 1 - \eta
\right)}{2N\Lavg}  \right)\left( F\left( x^k \right) -
F^* \right).
\]
\item For all $k \geq k_0$, where $k_0 \coloneqq \arg\min\{k:
F\left( x^k \right) - F^* < \Lmin R_0^2\}$, we have
\[
\E_{i_{k_0},\dotsc,i_{k-1}}\left[ F\left( x^k \right)
\mid x^{k_0} \right] -
F^* \leq \frac{2N\Lavg R_0^2}{2N + (1 - \eta)  (k - k_0)
}.
\]
\end{enumerate}
\end{enumerate}
\end{corollary}

The strategy \eqref{eq:lip} is referred to henceforth as ``Lipschitz
sampling.''  In {both Algorithms~\ref{alg:rcd1} and
\ref{alg:rcd2}, we  have}
\begin{equation*}
	\frac{\|H^k_i\|}{\alpha_i^k} = L_i.
\end{equation*}
{{Recalling the
    definitions of $M_i$ from \eqref{eq:H2} and $\cM$ from
	\eqref{eq:def.M}}, we have that since $H^k_i$ are fixed over $k$
	for all $i$, both algorithms have $\|H^k_i\| \equiv M_i$.
    Therefore,} \eqref{eq:lip} matches the
optimal probability distribution \eqref{eq:pi}, resulting in
\begin{equation} \label{eq:uf8}
\|\cP^{-1}\AA^{-1} \cM\| = N \Lavg.
\end{equation}

We next consider the case in which {the OSSC condition}
\eqref{eq:strong} holds for some $\mu>0$.
\begin{corollary}
\label{cor:rcd0}
Consider Algorithm \ref{alg:rcd1} and assume that \eqref{eq:compL}
holds.  For problems satisfying \eqref{eq:strong} with $\mu \in (0,
\Lmin]$,
{
the iteration complexity for the expected objective value to reach
\[
	\E_{i_0,\dotsc,i_{k-1}} \, F\left( x^k \right) - F^* 
	\leq \epsilon
\]
for any given $\epsilon > 0$} is as follows. {When $p^k_i
  \equiv 1/N$, $i=1,2,\dotsc,N$, we have complexity
  \[
  O\left(\frac{N\Lmax}{(1 - \eta) \mu} \log(1 / \epsilon) \right),
  \]
  while if the $p^k_i$ are defined by
  \eqref{eq:lip}, we have complexity
  \[
  O\left(\frac{N\Lavg}{(1 - \eta) \mu} \log (1 /
        \epsilon ) \right).  
        \]
  }
\end{corollary}
\begin{proof}
As shown in Lemma \ref{lemma:line2}, this choice of $H_i$ and
$\alpha_i$ satisfies \eqref{eq:usage} with $\gamma = 1$.  Thus the
case of uniform sampling is directly obtained from Theorem
\ref{thm:linear} and the known fact that for Q-linear convergence rate
of $1 - \tau$ with $\tau \in(0,1)$, the iteration complexity for
obtaining an $\epsilon$-accurate solution is $O(\tau^{-1}\log(1 /
\epsilon))$.

For \eqref{eq:lip}, we use \eqref{eq:Q} to derive a different result.
Since {$\|\cP_k^{-1}\AA_k^{-1}\cH_k\| = N\Lavg$ (from
  \eqref{eq:uf8})}, and by
letting $\lambda = 1/2$ and $\theta = \mu / (N\Lavg)$, \eqref{eq:Q}
leads to
\begin{equation}
	\E_{i_k}{\left[\left. \alpha_{i_k} Q^{k*}_{i_k}\right| x^k \right]}
	\leq \frac{\mu}{2 N \Lavg } \left( F^* - F\left( x^k \right)
	\right).
\end{equation}
  The
remainder of the proof tracks the proof of Theorem~\ref{thm:linear} to
get a Q-linear convergence rate.  \qed
\end{proof}

When $\eta= 0$ {(so that the solutions of the subproblems are
  exact)}, the rates in Corollaries~\ref{cor:rcd1} and \ref{cor:rcd0}
are similar to {Nesterov's result} \cite{Nes12a} for the
non-regularized case with the same sampling strategies, if we
interpret this result in the Euclidean norm.  The advantage of
Lipschitz sampling over uniform sampling is seen clearly.
Note that \cite{Nes12a} discusses the case of constrained
optimization, which can be treated as a special case of regularized
optimization. In this special case, Nesterov shows a $O(1/k)$
convergence rate of the objective value when the objective is convex,
but the convergence speed depends on $(R_0^2 / 2 + F(x^0) - F^*)$. Here, we weaken the dependency on the initial
objective value by showing linear convergence in the early stages of
iteration.  The case in which $F$ satisfies \eqref{eq:growth} can also
provide linear convergence for Algorithm~\ref{alg:rcd1}, but the
consequent rates do not suggest clear advantages of the Lipschitz
sampling, and the derivations are trivial. We therefore omit these
results.

When $f$ is not necessarily convex, Algorithm~\ref{alg:rcd1} still
benefits from Lipschitz sampling, as we now discuss.
\begin{corollary}
Consider Algorithm \ref{alg:rcd1} and assume that \eqref{eq:compL}
holds. {Suppose that a fixed probability distribution is used for the
choice of blocks, that is, $p^k_i \equiv p_i$ for all $k \ge 0$ and
all $i=1,2,\dotsc,N$.}
Then we have that
\begin{align*}
\min_{0 \leq k \leq T } \E_{i_0,\dotsc,i_{k-1}}
\left\|G_k\right\|^2
\leq \frac{2 (F\left( x^0 \right) - F^*)}{(1 - \eta) (T+1)}
\max_{1 \le i \le N}\frac{L_i}{p_i}.
\end{align*}
Therefore, when uniform sampling is used, we obtain
\begin{align*}
\min_{0 \leq k \leq T }
\E_{i_0,\dotsc,i_{k-1}} \left\|G_k\right\|^2
\leq \frac{2N \Lmax (F\left( x^0 \right) - F^*)}{(1 - \eta) (T+1)},
\end{align*}
whereas when Lipschitz sampling is used, we obtain
\begin{align*}
\min_{0 \leq k \leq T }
\E_{i_0,\dotsc,i_{k-1}} \left\|G_k\right\|^2
\leq \frac{2N \Lavg (F\left( x^0 \right) - F^*)}{(1 - \eta) (T+1)}.
\end{align*}
\end{corollary}
Our result here for the case of uniform sampling is similar to that in
\cite{PatN15a}, but we show that Lipschitz sampling can improve
the convergence rate by considering a slightly different measure of
stationarity.

We turn now to Algorithm \ref{alg:rcd2}, which can also be viewed as
an {extension of the algorithm in \cite{Nes12a} to the
  regularized problem \eqref{eq:f}.}
\begin{corollary}
\label{cor:rcd}
Consider Algorithm \ref{alg:rcd2} and assume that \eqref{eq:compL}
holds.
{Suppose that a fixed probability distribution is used for
  the choice of blocks, that is, $p^k_i \equiv p_i$ for all $k \ge 0$
  and all $i=1,2,\dotsc,N$.}  Then the following claims hold.
\begin{enumerate}
\item {For uniform sampling ($p_i = 1/N$, $i=1,2,\dotsc,N$)},
  we have
\begin{align*}
\min_{0 \leq k \leq T } \,
\E_{i_0,\dotsc,i_{k-1}} \, \left\|G_k\right\|^2 
\leq \frac{2N \Lmax (F\left( x^0 \right) - F^*)}{(1 - \eta) (T+1)}.
\end{align*}
\item If $f$ is convex, {then for uniform sampling, we have
  the following results.}
\begin{enumerate}
\item {When
  \begin{equation} \label{eq:jw8}
    F(x^k) - F^* \geq (1/\Lmax) (x^k -
    P_{\Omega}(x^k))^\top \cL (x^k - P_{\Omega}(x^k)),
  \end{equation}
  } where $\cL$ is defined in \eqref{eq:L}, the convergence of the
  expected objective value is Q-linear:
  \begin{align*}
	\E_{i_k}\left[F\left(x^{k+1}\right) - F^*\right|\left.x^k\right]
\leq \left( 1 - \frac{\left( 1 - \eta
	\right) }{2N \Lmax}\right) \left(F\left(
	x^k \right) - F^*\right).
\end{align*}

\item For all $k \geq k_0$, where $k_0 \coloneqq
	\arg\min\{k: F\left( x^k \right) - F^* < R_0^2\}$, the expected
	objective follows a sublinear convergence rate
\begin{align*}
\E_{i_{k_0},\dotsc,i_{k-1}} \left[ F\left(x^k\right) \right|
\left.x^{k_0}\right] - F^* \leq \frac{2N \Lmax R_0^2}{ 2N + (1 -
\eta)(k - k_0)}.
\end{align*}
\end{enumerate}
\item If $F$ satisfies {the OSSC condition} \eqref{eq:strong}
  for some $\mu > 0$, then for uniform sampling, we have
\begin{align*}
	\E_{i_k} \left[F(x^{k+1}) - F^* \right| \left.x^k \right]
\leq  \left( 1 -
\frac{(1 - \eta)(1 + 1 / \mu)^{-1}}{ N \Lmax } \right)
\left( F(x^k) - F^* \right).
\end{align*}
\item With $p_i$ chosen from \eqref{eq:lip}, results in
  {Parts 1 and 3 hold, with $\Lmax$ improved to $\Lavg$.
    For Part 2,  for convex $f$,
    we obtain the same improvement from $\Lmax$ to $\Lavg$ for all 
    rates, but the condition for early linear convergence becomes
    $F(x^k) - F^* \geq \|x^k - P_{\Omega}(x^k)\|^2$ rather than
    \eqref{eq:jw8}.}
\end{enumerate}
\end{corollary}

{Whether the OSSC condition
  \eqref{eq:strong} holds or not, the bounds indicate a potential
  improvement of $\Lmax/\Lavg$ in iteration bounds} when
\eqref{eq:lip} is used.

An advantage of Algorithm~\ref{alg:rcd1} over Algorithm~\ref{alg:rcd2}
is that when the solution exhibits some partial smoothness structure,
Algorithm~\ref{alg:rcd1} may be able to identify the low-dimensional
manifold on which the solution lies, as it is the case for the cyclic
variant described in \cite{Wri12a}. We can see that the convergence
rate bounds for $\|G_k\|$ are the same in both algorithms, and the
convergence in the general convex case after $k_0$ iterations is the
same as well, although the definition of $k_0$ can be different and
the early linear convergence conditions and rates also differ
slightly.
Thus, except when partial smoothness is present, the convergence
behaviors of the two algorithms appear to be similar.

\section{Related Work} \label{sec:related}

One of the (serial, deterministic) algorithms considered in our recent
paper \cite{LeeW18a} is a special case of Algorithm~\ref{alg:icd} with
only one block ($N=1$). The technique for measuring inexactness is
borrowed from \cite{LeeW18a}, but the extension described above, to
randomized BCD and arbitrary sampling probabilities, requires novel
convergence analysis.

The case in which \eqref{eq:quadratic} is solved exactly is discussed
in \cite{TseY09a}. This paper uses the same boundedness condition for
the $H^k_i$ as ours, and the blocks can be selected under a cyclic
manner (with an arbitrary order), or a Gauss-Southwell fashion.  For
the cyclic variant, the convergence rate of the special case in which
$Q$ forms an upper bound of the objective improvement is further
sharpened by \cite{SunH15a,LiZ16a}. The relaxation to approximate
subproblem solutions, with an inexactness criterion different from
ours, is analyzed in \cite{ChoPR16a}.  The latter paper shows linear
or sublinear convergence rates of a certain type, but the relation
between the convergence rates and either the measure of inexactness or
the choice of $H^k_i$ is unclear.  We note too that the cyclic
ordering of blocks is inefficient in certain cases: \cite{SunY16a}
showed that the worst case of cyclic BCD is $O(N^2)$ times slower than
the expected rate of randomized BCD.

The Gauss-Southwell variant discussed in \cite{TseY09a} can be
extended to the inexact case via straightforward modification of the
analyses for inexact variable-metric methods (see for example
\cite{LeeW18a,SchT16a,GhaS16a,BonLPP16a,PenZZ18a}), giving results
similar to what we obtain here with uniform sampling.  It may be
possible to utilize techniques for single-coordinate descent in
\cite{NutSLFK15a} to obtain better rates by considering a norm other
than the Euclidean norm, as was done in \cite{NutLS17a}, but such
extensions are beyond the scope of the current paper.

The special case of Algorithm~\ref{alg:rcd1} discussed in
Section~\ref{sec:rcd} has received much attention in the literature.
As mentioned earlier, the non-regularized case ($\psi\equiv 0$ in
\eqref{eq:f}) was first analyzed in \cite{Nes12a} for convex and
strongly convex $f$.  That paper uses a quadratic approximation of $f$
that is invariant over iterations, together with a fixed step
size. Since it is relatively easy to solve the subproblem to
optimality in the non-regularized case, inexactness is not considered.
The sampling strategy of using the probability $p_i = L_i^\alpha /
\sum_j L_j^\alpha$ for {any} $\alpha \in [0,1]$ was analyzed
in \cite{Nes12a}. The two extreme cases of $\alpha = 0$ and $\alpha=1$
correspond to uniform sampling and \eqref{eq:lip}, respectively. The
$i$th block update {in either case} is $d_i = -\nabla_i f(x)
/ L_i$, so we obtain from the blockwise Lipschitz continuity of
$\nabla f$ that
\begin{align*}
\E_i \left[f\left(x + U_i d_i\right) - f\left( x \right) \right]
&\leq \sum_{i} p_i f\left( x \right) - \frac{p_i}{2 L_i} \left\|\nabla_i
f(x)\right\|^2 - f\left( x \right)\\
&\leq -\min_{i} \frac{p_i}{2L_i}  \left\|\nabla f(x)\right\|^2.
\end{align*}
This bound suggests that if we use $p_i = 1/N$, the complexity will be
related to $N \Lmax$, whereas when $p_i$ is proportional to $L_i$, the
complexity is related to the smaller quantity $N \Lavg$, consistent
with our discussion in Section~\ref{sec:rcd}.  The case in which
$\psi$ is an indicator function of a convex set is also analyzed in
\cite{Nes12a}, with an extension in \cite{LuX15a} to convex and
strongly convex regularized problems, but both these analyses are
limited to Algorithm~\ref{alg:rcd1} with uniform sampling.  The case
in which $f$ {in \eqref{eq:f} is not necessarily convex} is
analyzed in \cite{PatN15a}, again under uniform sampling.
Our results allow broader choices of algorithm, and show that
non-uniform sampling can accelerate the optimization process.

{The special case of Algorithm~\ref{alg:rcd1} applied to the
  dual of convex regularized ERM, where each $\psi_i$ is strongly
  convex, with non-uniform samplings for the blocks, is analyzed in
  \cite{ZhaZ15a}. 
  Some primal-dual properties of these problems are
  used to derive the optimal probability distribution for the primal
  suboptimality.} It is unclear how to generalize this analysis to
other classes of problems.
Our recent work \cite{LeeW18c} shows a convergence rate of $o(1/k)$ of
Algorithm~\ref{alg:rcd1} when $f$ is convex, under arbitrary
non-uniform sampling of the blocks, and without the assumption of
finite $R_0^2$.  However, this work does not show convergence
improvement for non-uniform sampling, like the improvement shown above
for \eqref{eq:lip}.  Moreover, our earlier paper does not address the
early linear convergence rates in the convex case.

{He et al. \cite{HeTT18a} consider the case of adaptive
  probability distributions that change every iteration for sampling
  the coordinates or the blocks, for an algorithm slightly different
  from the BCD framework considered here.  They show that suitable
  choices for adaptive probabilities may further improve the
  convergence.  Although our framework allows for adaptive probability
  distributions as well, most of our convergence results are for fixed
  probabilities. Moreover, most works considering adaptive
  probabilities do not yield an empirical advantage for the adaptive
  distribution that give better theoretical convergence, because
  updating the probabilities followed by sampling can incur an
  additional per-iteration cost of $O(N)$ (and a cost of $O(N^2)$ per
  ``epoch'' of $n$ successive iterations).  For high-dimensional
  problems, these works usually rely on heuristics to work in
  practice; see the discussion in \cite{HeTT18a} and the references
  therein.}

The paper \cite{TapRG16a} describes inexact extensions of
\cite{Nes12a} to convex versions of \eqref{eq:f}. This paper uses a
different inexactness criterion from ours, and their framework fixes
$H^k_i$ over all iterations, using small steps based on $L_i$ rather
than a line search.  Thus, their algorithm requires knowledge of the
parameters $L_i$.  In the regularized case of $\psi \neq 0$, their
algorithm is compatible only with uniform sampling.  \cite{FouT15a}
allows variable $H_i$ and backtracking line search, but under a
different sampling strategy in which a predefined number of blocks is
sampled at each iteration from a uniform distribution. The other
difference between our algorithm and that of \cite{FouT15a} is that
their inexactness condition can be expensive to check except for
special cases of $\psi$ (see their Remark 5). Our improvements over
\cite{FouT15a} include (1) an inexactness framework that allows more
general $\psi$, (2) non-uniform sampling that may lead to significant
acceleration when additional information is available, (3) sharper
convergence rates, and (4) convergence rate results for nonconvex $f$.

\section{Efficient Implementation for Algorithm~\ref{alg:icd}}
\label{sec:erm}

An important concern in assessing the practicality of
Algorithm~\ref{alg:icd} is whether the operations of partial gradient
evaluation and line search can be carried out efficiently, and whether
there are natural choices of the variable metrics $H^k_i$ that can be
maintained efficiently. In this section and the computational section
to follow, we consider problems in which $f$ has the form
\begin{equation}
	f(x) = g(Ax)
\label{eq:erm}
\end{equation}
for a given matrix $A\in \R^{\ell \times n}$ and a function $g:\R^\ell
\rightarrow \R$ that is block-separable, and the evaluation of $g(z)$
costs $O(\ell)$ operations.  This structure includes many problems
seen in applications, including the regularized ERM problem in machine
learning and its Lagrange dual.  We also discuss the practicality of
non-uniform sampling in this section.

One key to efficient implementation of Algorithm~\ref{alg:icd} is to
maintain explicitly the matrix-vector product $Ax$, updating it during
each step. The updates have the form
\[
A(x+U_i d_i) = Ax + A_{i} d_i,
\]
where $d_i \in \R^{n_i}$ is the update to the $i$th block and $A_i
\coloneqq AU_i$ is the column submatrix of $A$ that corresponds to
this block. The partial gradient has the form
\[
\nabla_i f(x) = A_i^\top \nabla g(Ax),
\]
so it can be evaluated at the cost of evaluating $\nabla g$ (costs
$O(\ell)$ operations as evaluating $g$ costs $O(\ell)$) together with
a matrix-vector product involving $A_i$.

To perform the line search in Algorithm~\ref{alg:icd}, we need to
evaluate $\psi_i(x_i+\alpha d_i)$ for each value of $\alpha$, along
with $f(x+\alpha U_i d_i) = g(Ax+\alpha A_i d_i)$. Once $A_i d_i$ has
been calculated (once), the marginal cost of performing this operation
for each $\alpha$ is the $O(\ell)$ operations needed to calculate
$Ax+\alpha A_i d_i$ and the $O(\ell)$ operations needed to evaluate
$g$.

A natural choice for the {quadratic} term $H^k_i$ in subproblem
\eqref{eq:Qdef} is the $i$th diagonal block of the true Hessian, which
is
\begin{equation} \label{eq:ermH}
    [\nabla^2 f(x)]_{ii} = A_i^\top \nabla^2 g(Ax) A_i.
\end{equation}
{(Note that the subscript is the $(i,i)$ block, not the $(i,i)$
  entry.)}  The block-separability of $g$ makes $\nabla^2 g(Ax)$
block-diagonal, and actually diagonal in many applications. Thus the
matrix \eqref{eq:ermH} has a particularly simple form. We note
moreover that when iterative methods are used to (approximately)
minimize \eqref{eq:Qdef}, we do not need to know this matrix
explicitly, but only to be able to compute matrix-vector products of
the form $H^k_i v_i$ (for various $v_i$) efficiently. This operation
can be done at the cost of two matrix-vector multiplications involving
$A_i$, together with the (typically $O(\ell)$) cost of multiplying by
$\nabla^2 g(Ax)$.

There are two concerns in implementing non-uniform samplings such as
the Lipschitz sampling.  The first is simply the cost of sampling from
a non-uniform distribution, for which a naive method may cost $O(N)$
operations.  Fortunately, there are efficient methods such as that
proposed in \cite{Wal77a} for non-uniform samplings such that given a
fixed distribution, after a $O(N)$ cost of initialization, each run
costs the same as sampling two points uniformly randomly.  Note that
the overhead incurred in changing probability distributions $\{ p^k_i
\}$ between iterations can nullify any efficiencies gained; the
sampling can then become the bottleneck especially when the update itself
is inexpensive.  For completeness, we give details of our
implementation of non-uniform sampling in Appendix~\ref{app:sample}.

The second concern is that the cost per iteration is different under
different sampling strategies. Especially when the data are sparse,
the value of $L_i$ may be positively correlated to the density of the
corresponding data point.  In this case, sampling according to $L_i$
may increase the cost per iteration significantly.  However, if one
can estimate each norm $\|H_i\|$, the step sizes, and the cost of
updating different blocks in advance, it is not hard to compare the
expected cost increase and the expected convergence improvement to
decide if non-uniform sampling should be considered.  When such
information is unavailable or hard to obtain, uniform sampling can
still be used.

\section{Computational Results} \label{sec:exp}

\begin{figure}
\centering
\begin{tabular}{@{}ccc@{}}
	\includegraphics[width=0.31\linewidth]{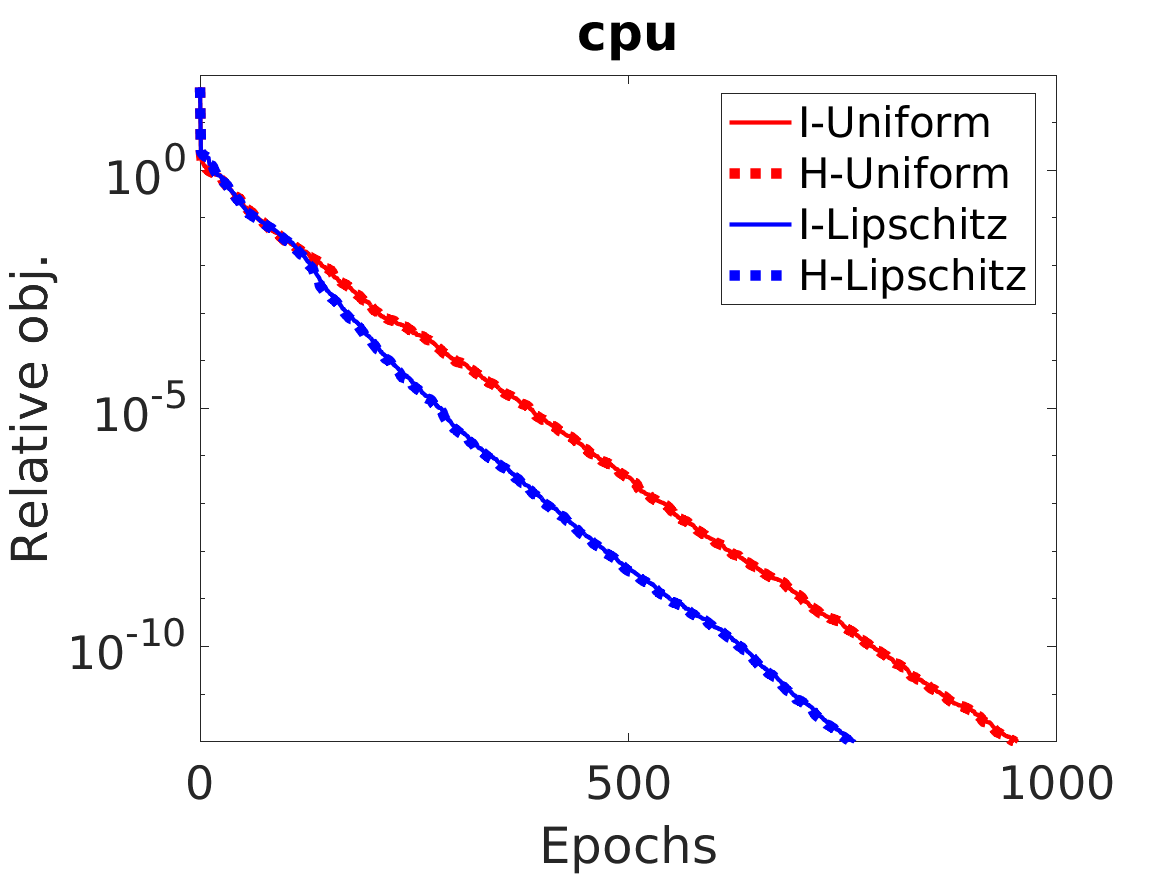}&
	\includegraphics[width=0.31\linewidth]{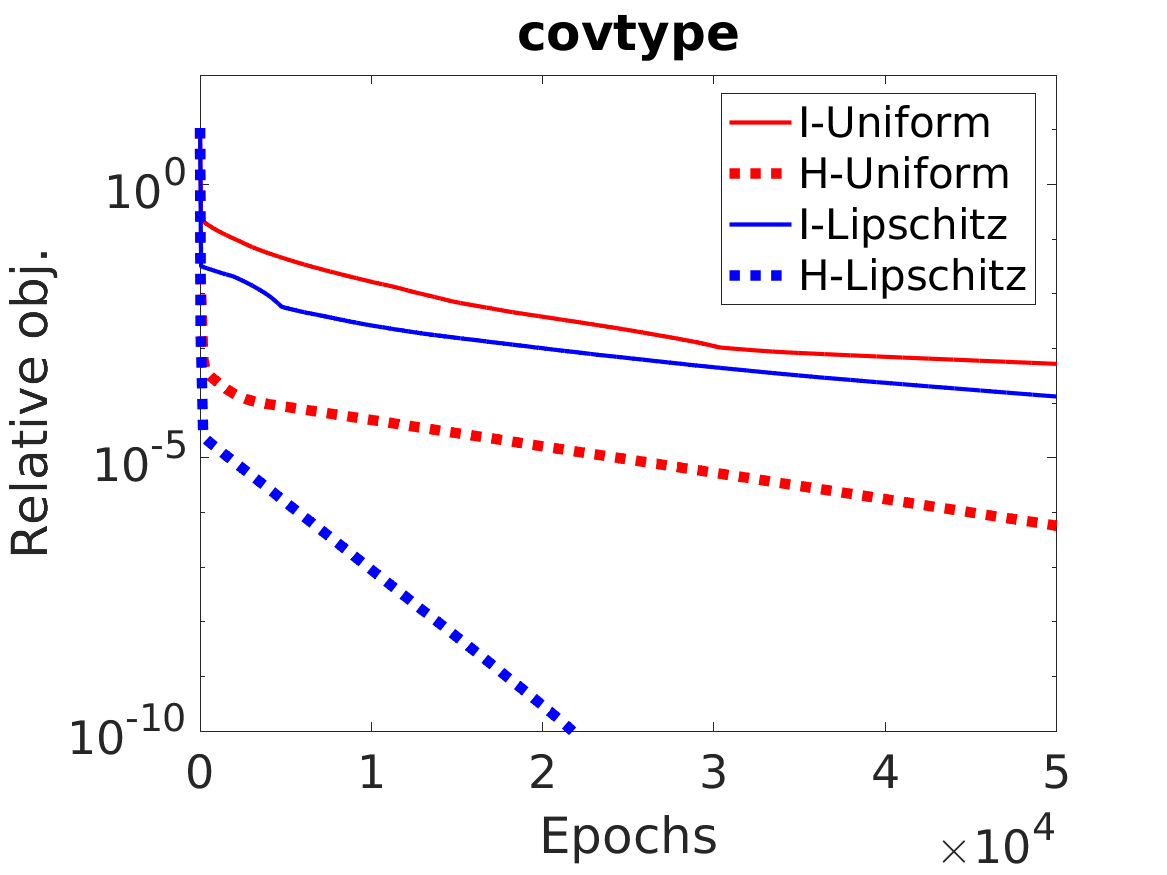}&
	\includegraphics[width=0.31\linewidth]{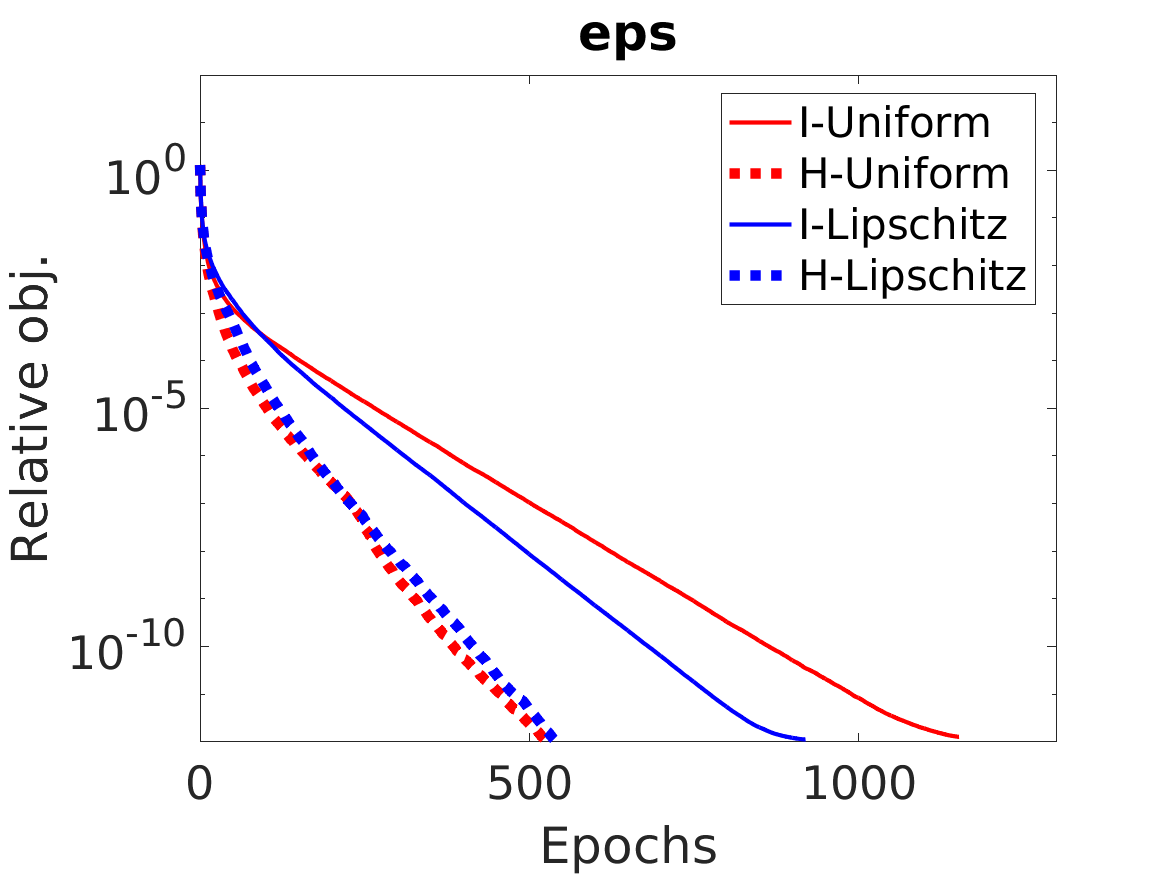}
\end{tabular}
\caption{Comparison of different sampling strategies using fixed step
  sizes in terms of epochs.  The prefix ``H'' refers to the choice
  $H_i = L_i I$, while ``I'' means $H = I$.}
\label{fig:sampling}
\end{figure}

\begin{figure}
\centering
\begin{tabular}{@{}cc@{}}
\includegraphics[width=.48\linewidth]{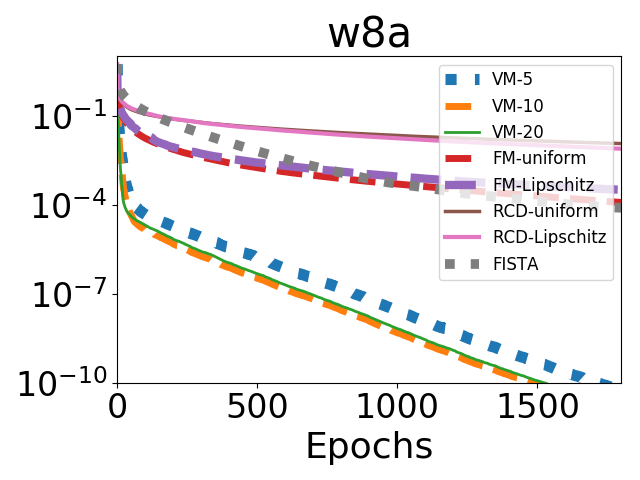}&
\includegraphics[width=.48\linewidth]{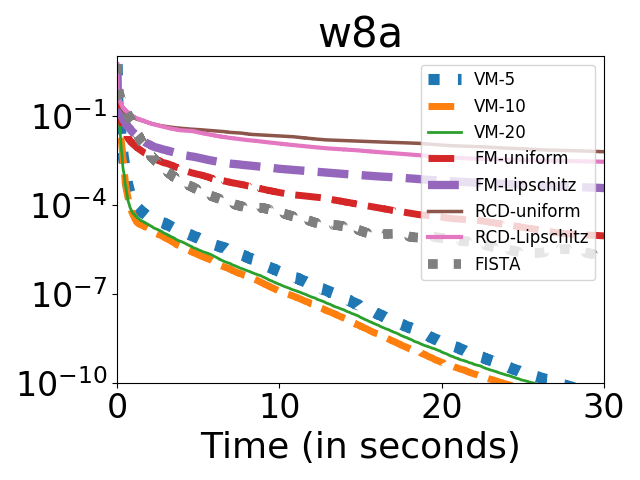}\\
\includegraphics[width=.48\linewidth]{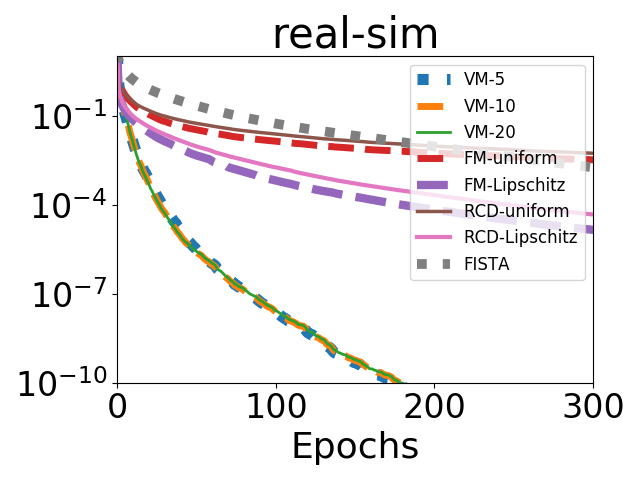}&
\includegraphics[width=.48\linewidth]{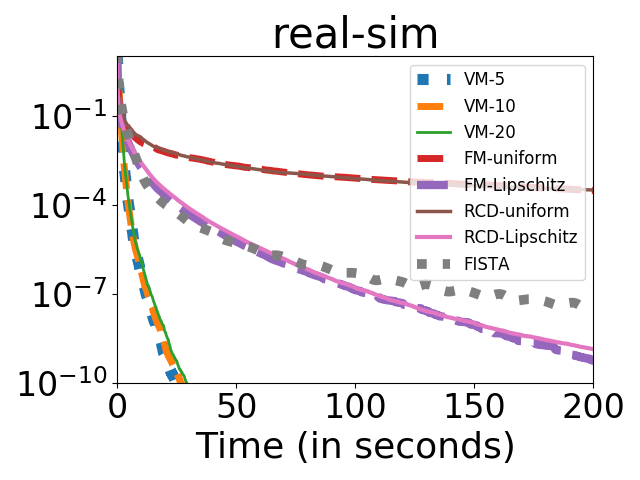}\\
\includegraphics[width=.48\linewidth]{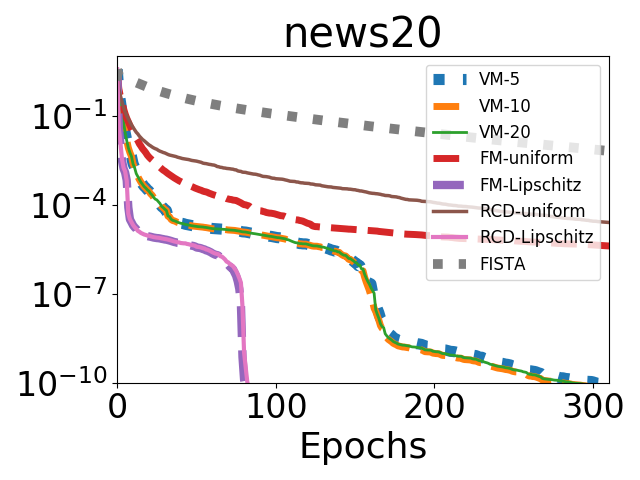}&
\includegraphics[width=.48\linewidth]{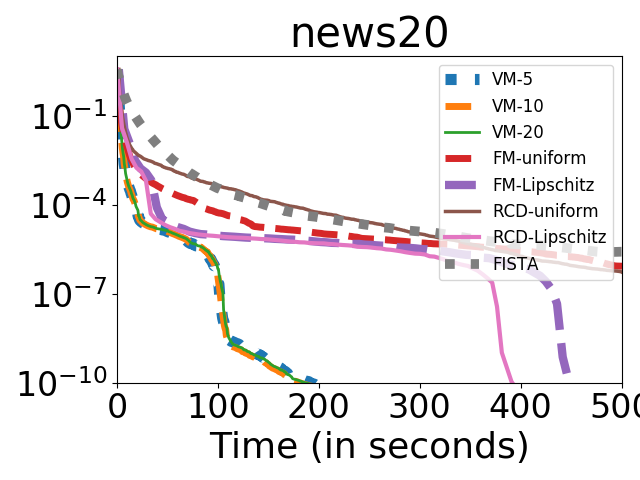}\\
\includegraphics[width=.48\linewidth]{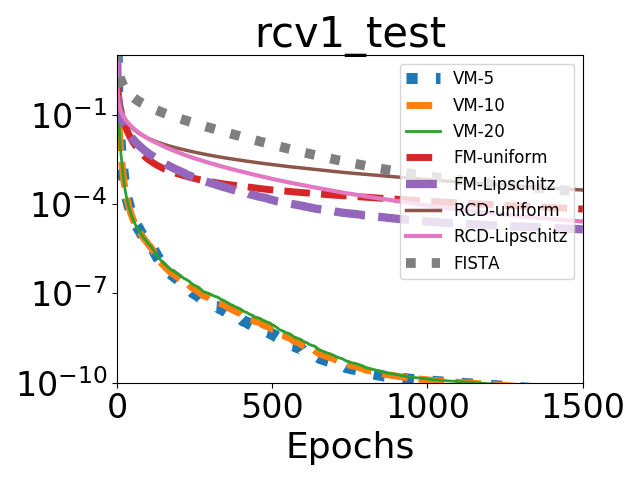}&
\includegraphics[width=.48\linewidth]{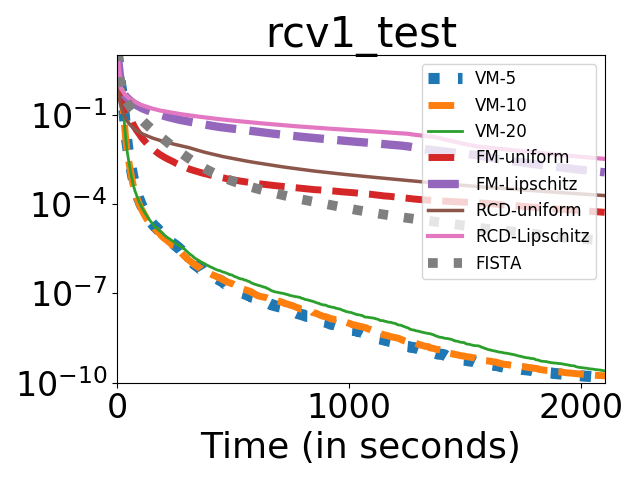}
\end{tabular}
\caption{Comparison of fixed  and variable quadratic terms for solving
	\eqref{eq:l2} with $C=1$.
	{Left column: epochs, right column}: running time.}
\label{fig:l2}
\end{figure}

This section reports on the empirical performance of
Algorithms~\ref{alg:icd}-\ref{alg:rcd2} on {three} sets of experiments.
In the first set of problems, which are convex, we compare uniform
sampling and the Lipschitz sampling for the traditional randomized BCD
approaches discussed in Section~\ref{sec:rcd}, on both
Algorithms~\ref{alg:rcd1} and \ref{alg:rcd2}. In the second set of
experiments, also on convex objectives, we investigate a version of
Algorithm~\ref{alg:icd} in which the $i$th diagonal block of the true
generalized Hessian is used as $H^k_i$ in \eqref{eq:Qdef}.  In both experiments,
we report the relative objective value difference to the optimum,
defined as $(F(x) - F^*)/F^*$, where $F^*$ is obtained by running our
algorithm with a tight termination condition.  { The third
  set of experiment considers a nonconvex problem, and therefore the
  algorithms are not guaranteed to find $F^*$.  We report the measure
  $\|G_k\|^2$ of stationarity instead.}

\subsection{Traditional Coordinate Descent}

\begin{table}
	\centering
	\begin{tabular}{@{}l|rr|c|r}
		Data set & \#instances & $n$ & $\Lmax/\Lavg$ & $C$\\
		\hline
		cpusmall\_scale & $8,192$ & $12$ &$1.29$ & $.001$\\
		covtype.binary.scale & $581,012$ & $54$ & $8.58$ & $.001$\\
		epsilon\_normalized & $400,000$ & $2,000$ &$5.49$ & $.2$
	\end{tabular}
	\caption{Data sets used in the LASSO problem.}
	\label{tbl:lasso}
\end{table}

We first illustrate the speedup of Lipschitz sampling over uniform
sampling using the simple LASSO problem \citep{Tib96a}
\begin{equation}
	\min_{x\in \R^n} \quad \frac{C}{2} \sum_{i=1}^l (a_i^\top  x -
	b_i)^2 + \|x\|_1,
	\label{eq:lasso}
\end{equation}
where $(a_i,b_i) \in \R^n \times \R, i=1,\dotsc, l,$ are the training
data points and $C>0$ is a parameter to balance the two terms.
{In the subproblem, each ``block'' consists of only one coordinate and
therefore $n = N$.}
Note that the corresponding
subproblem \eqref{eq:quadratic} has a closed-form solution when $H$ is
a multiple of identity, so we have $\eta = 0$ in \eqref{eq:approx}.

Our goal here is not to propose an optimal BCD algorithm for
\eqref{eq:lasso} but merely to compare sampling strategies.
We choose $C$
so that among the final solutions generated by different variants we
compare, the sparsest one has a sparsity of around $50\%$.
Statistics of the data sets and the value of $C$ are listed in
Table~\ref{tbl:lasso}.  We test both Algorithms~\ref{alg:rcd1} and
\ref{alg:rcd2}, and both uniform and Lipschitz samplings.  We present
convergence in terms of epochs, where each epoch is a group of
{$N$ successive iterations.}  Most of the results in
Figure~\ref{fig:sampling} show a clear advantage for Lipschitz
sampling, consistent with our convergence analysis. The only exception
is Algorithm~\ref{alg:rcd2} on the data set epsilon, where the two
sampling strategies give similar performance.  The major reason for
this exception is that different sampling strategies identified the
correct active set at different stages, and these differences affect
the overall convergence behavior.

We also observe that {because of the effects of active set
  identification, Algorithm~\ref{alg:rcd1} often outperforms
  Algorithm~\ref{alg:rcd2}, but when $n$ is small (as in
  cpusmall\_scale), the two perform quite similarly.}  {Early
  fast convergence can be observed empirically in all examples, as
  suggested by Theorem~\ref{thm:sub}.}

\subsection{Variable Metric Approach}

We show the advantage of using variable quadratic terms $H^k_i$ in
\eqref{eq:Qdef}, in comparison with a fixed term.  For this purpose,
we consider a group-LASSO regularized squared-hinge loss problem
defined by
\begin{equation}
	\min_{x\in \R^n} \quad C\sum_{i=1}^l \max\left\{1 - b_i a_i^\top  x, 0
	\right\}^2 + \sum_{i=1}^{\lceil n /5\rceil}\sqrt{\sum_{j=1}^{\min\{5, n -
	5(i-1)\}} x_{5(i-1) + j}^2},
	\label{eq:l2}
\end{equation}
where $(a_i,b_i) \in \R^n \times \{-1,1\}, i=1,\dotsc, l$ are the
training data points and $C>0$ is a parameter to balance the two
terms.  Each set of five consecutive coordinates is grouped into a
single block to form the regularizer.  We compare the following
algorithms.
\begin{itemize}
\item VM-$t$: our variable metric approach of Algorithm~\ref{alg:icd},
  with $H$ being the {generalized Hessian} with $10^{-10} I$
  added to ensure that the condition \eqref{eq:H2} is satisfied
  {with $m_i>0$}.  We use uniform sampling of the blocks and
  the SpaRSA approach of \cite{WriNF09a} to solve the subproblem, with
  $t \in \{5,10,20\}$ being the number of SpaRSA iterations applied to
  each subproblem.
\item FM: the fixed metric approach considered in \cite{TapRG16a}.  We
  use a global upper bound of the {generalized} Hessian as
  the fixed metric.  As $H_i$ are precomputed, we consider both
  uniform sampling and the sampling scheme of \eqref{eq:pi}
  {using the largest eigenvalue of each $H_i$}.  {We
    solve each subproblem inexactly using $10$ SpaRSA iterations.}
\item RCD: Algorithm~\ref{alg:rcd1} with $\eta = 0$.  We use both
  Lipschitz sampling \eqref{eq:lip} and uniform sampling.
\item FISTA \cite{BecT09a}: the accelerated proximal gradient approach
  that does not exploit the block-separable nature of the
  regularization term.
\end{itemize}
The FISTA approach is included as a comparison with state of the art
for problems without block separability.

We consider the data sets in Table~\ref{tbl:l2}, obtained from the
LIBSVM
website,\footnote{\url{https://www.csie.ntu.edu.tw/~cjlin/libsvmtools/datasets/}.}
and set $C=1$ in \eqref{eq:l2}.  Results are shown in
Figure~\ref{fig:l2}.  Note that the varying number of SpaRSA
iterations used in VM-5, VM-10, and VM-20 have little impact on the
convergence in terms of both epochs and running time, and that all
these variants are significantly faster than their competitors,
showing the advantages of solving the subproblems with variable
metrics inexactly.
For news20, {Lipschitz sampling with both the fixed metric
  approach and Algorithm~\ref{alg:rcd1}} are the fastest in terms of
epochs, but the running times are much slower than the proposed
variable metric approach.  The reason is that news20 is a very sparse
data set, with the size of the Lipschitz constants highly correlated
to the density of each coordinate, making the average number of nonzero
elements processed per epoch much higher when Lipschitz sampling is
considered.

We also observe that for both the fixed metric approach and
Algorithm~\ref{alg:rcd1}, Lipschitz sampling is always faster than
uniform sampling in terms of epochs, confirming our analysis.  But in
terms of running time, the situation may differ.  We also observe that
FISTA performs better in running time than in epochs, mainly because
it updates the variables and the gradient less frequently, and its
memory access is always sequential and therefore faster.  Finally, we
observe the early linear convergence in the variable metric approach,
the fixed metric approach, and Algorithm~\ref{alg:rcd1}, verifying the
result in Theorem~\ref{thm:sub} empirically.

We also notice that although the variable metric approach is the only
one that requires line search, it is still the fastest in terms of
running time, showing that line search does not occupy a significant
portion of the running time.

\begin{table}
	\centering
	\begin{tabular}{l|rrr}
		Data set & \#instances & $n$\\
		\hline
		w8a & $49,749$ & $300$\\
		real-sim & $72,309$ & $20,958$\\
		news20 & $ 19,996$ & $1,355,191$\\
		rcv1\_test & $677,399$ & $47,236$
	\end{tabular}
	\caption{Data sets used in the group-LASSO regularization experiment.}
	\label{tbl:l2}
\end{table}

{
\subsection{A Nonconvex Problem}
We now consider a nonconvex problem.
Following the setting of \cite{CarDHS17a}, we consider the smooth
biweight loss by \cite{BeaT74a}:
\begin{equation}
	f(x) = C \sum_{i=1}^l \phi\left( a_i^\top x - b_i \right), \text{
	where } \phi\left( z \right) = \frac{z^2}{1 + z^2},
	\label{eq:biweight}
\end{equation}
for some $C>0$, with $(a_i, b_i) \in \R^n \times \R$ for
$i=1,\dotsc,l$.
Through simple calculation, we can see that the Hessian of $\phi$
is
\[
	\phi''(z) = -\frac{2 \left( 3 z^2 - 1 \right)}{(z^2+1)^3}.
\]
Its value lies in $[-0.5, 2]$, showing that $f$ is nonconvex.
For the regularization term,
we consider both the $\ell_1$ norm used in the first set of experiments
and the group-LASSO regularization used in the second set of
experiments.}

{ For the $\ell_1$-regularized problem, we compare different
  sampling strategies of RCD.  In the previous results,
  Algorithm~\ref{alg:rcd1} tends to perform better than
  Algorithm~\ref{alg:rcd2}, so we apply only the former in this
  experiment.  As this nonconvex problem is harder than LASSO, we
  consider the first two smaller data sets in Table~\ref{tbl:lasso}.
  Results are shown in Figure~\ref{fig:l1biweight}.  We see that as
  predicted by our theory, sampling according to \eqref{eq:lip} yields
  faster convergence than uniform sampling.}

\begin{figure}
\centering
\begin{tabular}{cc}
	Epochs & Time\\
	\multicolumn{2}{c}{cpusmall\_scale }\\
	\includegraphics[width=0.45\linewidth]{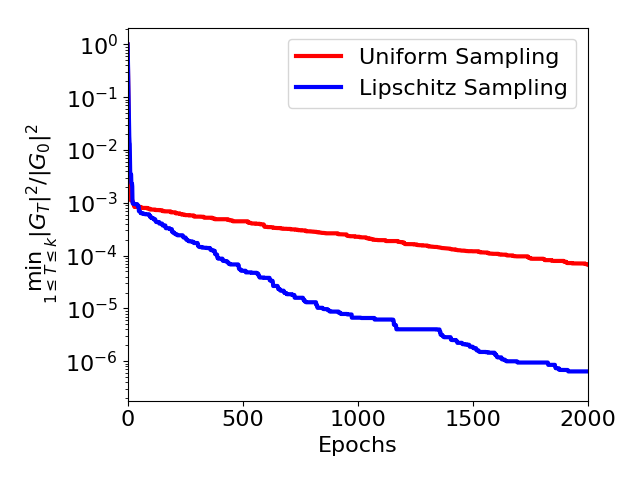}&
	\includegraphics[width=0.45\linewidth]{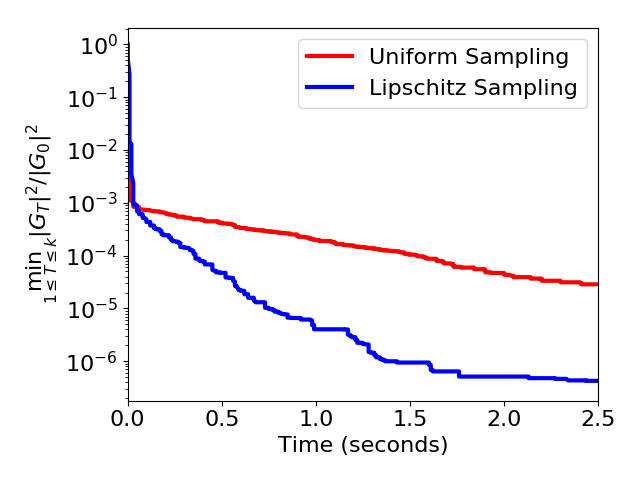}\\
	\multicolumn{2}{c}{covtype.binary.scale}\\
	\includegraphics[width=0.45\linewidth]{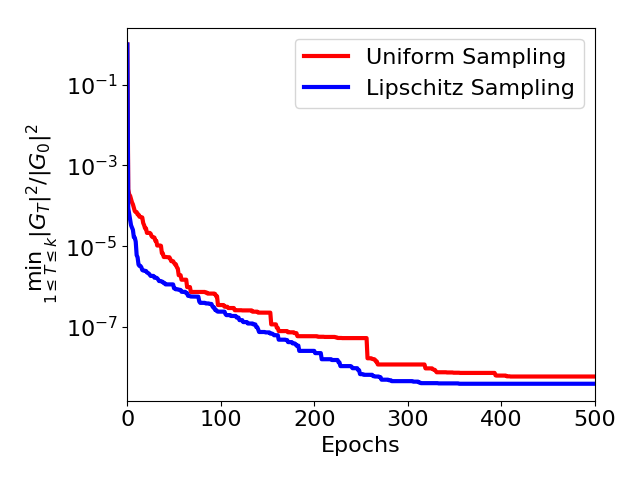}&
	\includegraphics[width=0.45\linewidth]{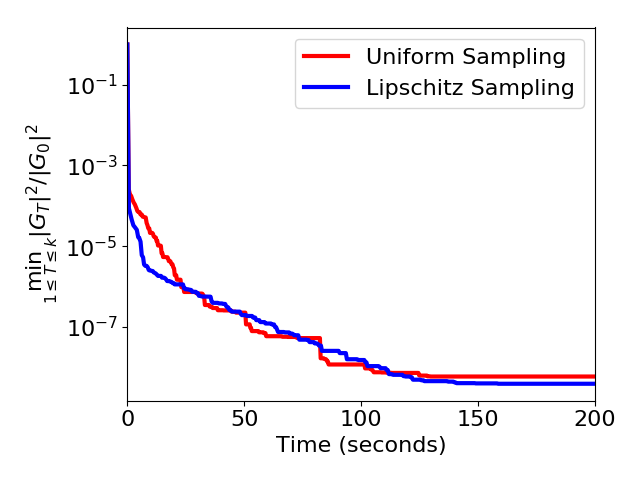}
\end{tabular}
\caption{Comparison of different sampling strategies on the nonconvex
biweight problem with $\ell_1$ regularization.  }
\label{fig:l1biweight}
\end{figure}

{ For the group-LASSO-regularized part, different from the
  previous experiment, we do not include FISTA in our comparison
  because it is not applicable to nonconvex problems.  The FM approach
  obtains the global upper bound for the Hessian through using the
  upper bound $2$ for $\phi''(a_i^\top x - b_i)$ for all $i$.  For the VM
  approach, the Hessian block may be indefinite so we obtain $H^k_i$
  by adding a multiple of identity  as needed to make it positive
  definite.
  In particular, we compute the eigenvalues of the Hessian block, and
  when the smallest eigenvalue is smaller than $10^{-10}$, we add a
  multiple of identity to $H^k_i$ to make the smallest eigenvalue
  exactly $10^{-10}$, and otherwise we do not modify $H^k_i$.
  Note that since the size of each $H^k_i$ is at most $5 \times 5$,
  computing its eigenvalues is cheap.
 We
  conduct the comparison using the first three data sets in
  Table~\ref{tbl:l2}.  The comparison between the variable metric
  approach and the fixed metric approach with different samplings is
  shown in Figures~\ref{fig:groupbiweight}.
  All approaches use $10$ SpaRSA iterations for each subproblem.  On
  all three data sets, the variable metric approach converges faster
  than the fixed metric approach with uniform sampling. As in the
  previous experiment, Lipschitz sampling has much better convergence
  on news20 in terms of epochs.  An interesting difference is that
  Lipschitz sampling does not work well on the other two data sets.  A
  further examination indicates that on those two data sets, the
  Lipschitz sampling strategy identifies the correct sparsity pattern
  much later, possibly affecting the convergence behavior.  }

{With regard to running time, the fixed metric approach with
  uniform sampling tends to be the fastest.  The reason is that on
  this nonconvex problem, the convergence advantage of the variable
  metric approach is not significant enough to counterbalance the
  higher per-iteration cost.  The less strong convergence advantage is
  likely from the damping term being added to the variable metric.
  There are various ways to modify the indefinite Hessian to make it
  positive definite, but so far there is no conclusion which approach
  is most effective.  Comparing various Hessian modification
  strategies is an interesting future work.
}

{This set of experiments shows that when we are dealing with
  nonconvex problems, variable metric approach based on the Hessian
  might be less effective because of the indefiniteness of the
  Hessian.  On the other hand, Lipschitz sampling has better
  convergence speed on three out of the five data sets, indicating
  that when the sparsity pattern identification is not a problem,
  Lipschitz sampling has better convergence speed.}

\begin{figure}
\centering
\begin{tabular}{cc}
	Epochs & Time\\
	\multicolumn{2}{c}{w8a}\\
	\includegraphics[width=0.45\linewidth]{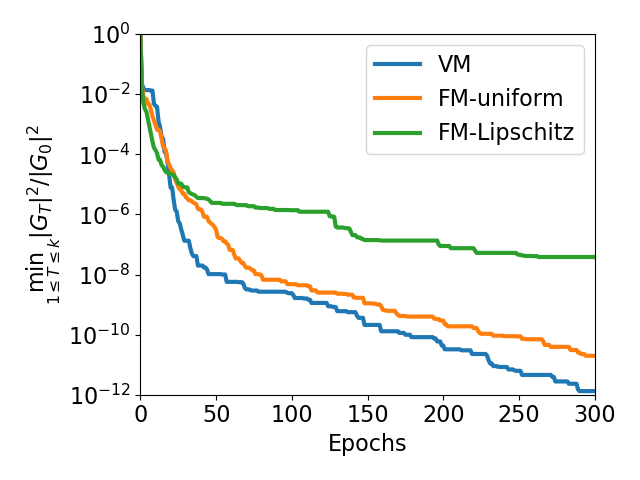}&
	\includegraphics[width=0.45\linewidth]{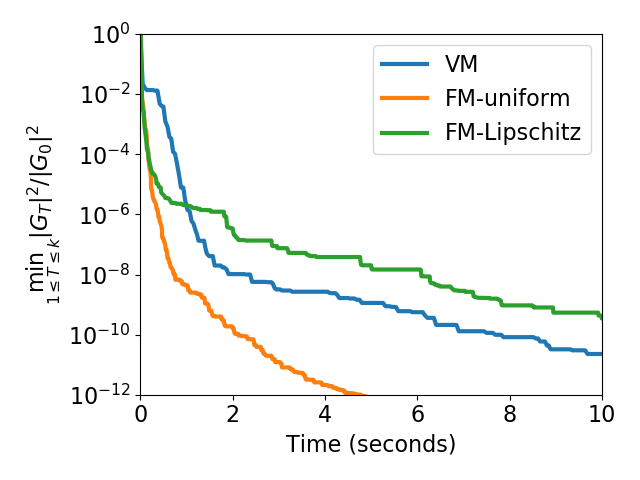}\\
	\multicolumn{2}{c}{real-sim}\\
	\includegraphics[width=0.45\linewidth]{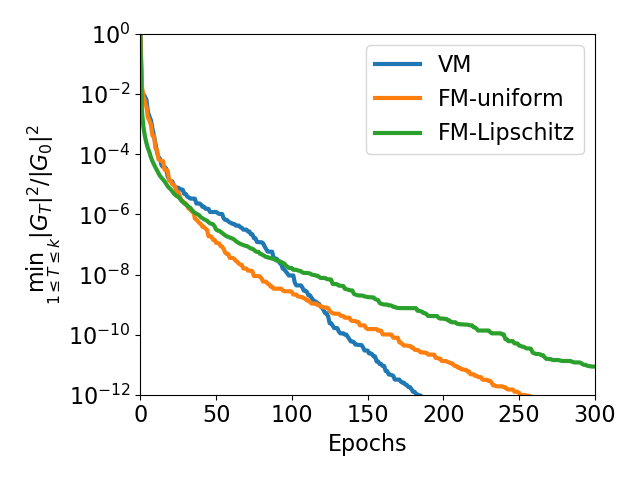}&
	\includegraphics[width=0.45\linewidth]{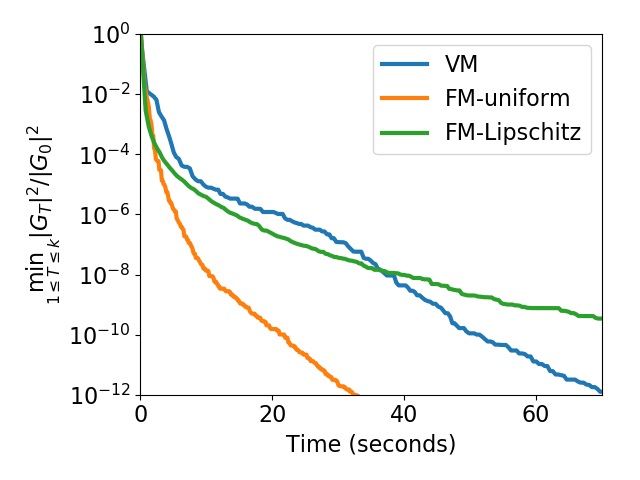}\\
	\multicolumn{2}{c}{news20}\\
	\includegraphics[width=0.45\linewidth]{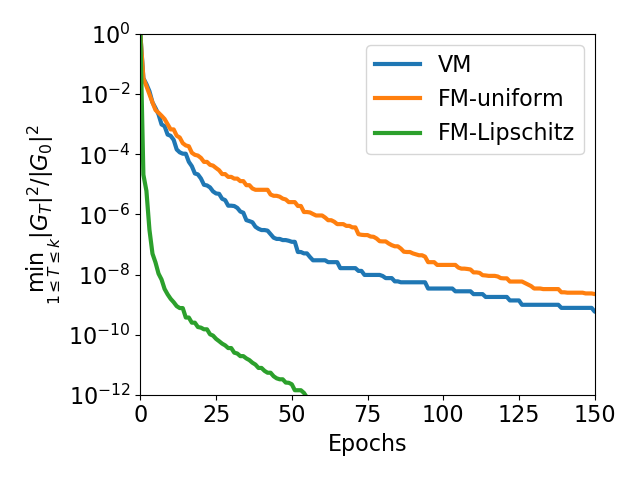}&
	\includegraphics[width=0.45\linewidth]{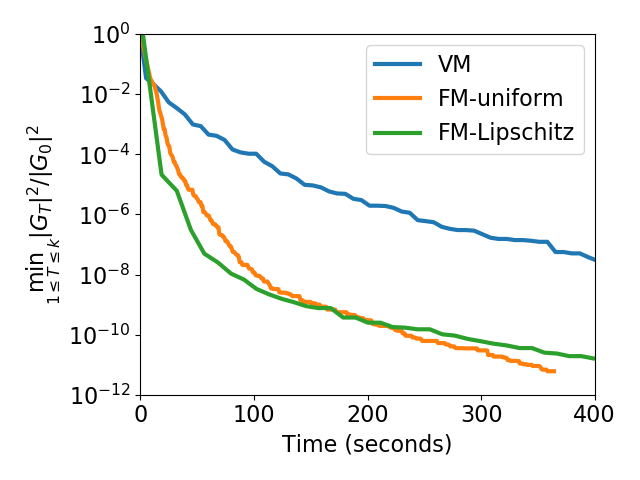}\\
\end{tabular}
\caption{Comparison of fixed  and variable quadratic terms for solving
the biweight loss problem with group-LASSO regularization. The
$y$-axis is $\min_{0 \leq k \leq T} \|G_k\|^2 / \|G_0\|^2$.}
\label{fig:groupbiweight}
\end{figure}

\section{Conclusions}
\label{sec:conclusion}
Starting with a strategy for regularized optimization using
regularized quadratic subproblems with variable quadratic terms, we
have described a stochastic block-coordinate-descent scheme that is
well suited to large-scale problems with general structure. We provide
detailed iteration complexity analysis, allowing for arbitrary
sampling schemes.  A special case of our theory extends known results
for a sampling strategy based on blockwise Lipschitz constants for
randomized gradient-coordinate descent from the non-regularized
setting to the regularized problem \eqref{eq:f} {and from
  convex problems to nonconvex problems}.  Computational experiments
show empirical advantages for our variable metric approaches.

\bibliographystyle{spmpsci}
\bibliography{icd,../inexactprox}
\appendix
\section{Efficient Implementation of Nonuniform Sampling}
\label{app:sample}
We describe our implementation of non-uniform sampling.  The $O(N)$
initialization step is described in Algorithm~\ref{alg:init}.  After
the initialization, each time to sample a point from the given
probability distribution, it takes only $2$ independent uniform
sampling as described in Algorithm~\ref{alg:sample}.

\begin{algorithm}
\caption{Initialization for non-uniform sampling}
\label{alg:init}
\begin{algorithmic}[1]
\STATE Given a probability distribution $p_1,\dotsc, p_N > 0$;
\STATE $i \leftarrow 1$;
\STATE Construct $U \leftarrow \{u \mid p_u > 1/ N\}$, $L
\leftarrow\{l \mid p_l \le 1 / N\}$;
\WHILE{$L \neq \phi$}
	\STATE Pop an element $l$ from $L$;
	\STATE Pop an element $u$ from $U$;
	\STATE $\text{upper}_i \leftarrow u$, $\text{lower}_i \leftarrow
	l$, $\text{threshold}_i \leftarrow p_l / (1 / N)$;
	\STATE $p_u \leftarrow p_u - (1 / N - p_l)$;
	\IF {$p_u > 1 / N$}
		\STATE $U \leftarrow U \cup \{u\}$;
	\ELSE
		\STATE $L \leftarrow L \cup \{u\}$;
	\ENDIF
	\STATE $i \leftarrow i + 1$;
\ENDWHILE
\end{algorithmic}
\end{algorithm}

\begin{algorithm}
\caption{Nonuniform sampling after initialization by Algorithm \ref{alg:init}}
\label{alg:sample}
\begin{algorithmic}[1]
\STATE Sample $i$ and $j$ independently and uniformly from
$\{1,\dotsc,N\}$;
\IF{$j / N \ge \text{threshold}_i$}
	\STATE Output $\text{upper}_i$;
\ELSE
	\STATE Output $\text{lower}_i$;
\ENDIF
\end{algorithmic}
\end{algorithm}

\end{document}